\documentclass[11pt]{amsart}
\usepackage{amsmath,amssymb,amsthm}
\usepackage[letterpaper, margin=1in]{geometry}
\usepackage{url,hyperref}
\hypersetup{
    colorlinks=true,
    linkcolor=black,
    citecolor=black,
    urlcolor=black
}
\usepackage{enumitem}
\usepackage[all]{xy}
\usepackage{tikz}
\usetikzlibrary{decorations.pathmorphing}
\usepackage{graphicx, import, calc} 
\usepackage{verbatim}

\usepackage{stmaryrd}

\DeclareFontFamily{U}{mathc}{}
\DeclareFontShape{U}{mathc}{m}{it}%
{<->s*[1.03] mathc10}{}

\DeclareMathAlphabet{\mathcal}{U}{mathc}{m}{it}

\newtheorem{theorem}{Theorem}[section]

\newtheorem{remark}[theorem]{Remark}

\newtheorem{proposition}[theorem]{Proposition}
\newtheorem{lemma}[theorem]{Lemma}

\newtheorem{assumption}[theorem]{Assumption}

\newtheorem{example}[theorem]{Example}

\renewcommand{\tilde}{\widetilde}
\renewcommand{\hat}{\widehat}
\newcommand{\Q}{\mathbb{Q}}
\newcommand{\R}{\mathbb{R}}
\newcommand{\Z}{\mathbb{Z}}

\newcommand{\C}{\mathbb{C}}
\newcommand{\T}{\mathcal{T}}
\newcommand{\B}{\mathcal{B}}
\renewcommand{\H}{\mathcal{H}}
\renewcommand{\phi}{\varphi}
\newcommand{\A}{\mathcal{A}}

\renewcommand{\P}{\mathbb{P}}
\renewcommand{\O}{\mathcal{O}}
\newcommand{\X}{\mathcal{X}}
\newcommand{\Y}{\mathcal{Y}}
\newcommand{\I}{\mathcal{I}}
\renewcommand{\L}{\mathcal{L}}
\newcommand{\V}{\mathcal{V}}
\newcommand{\M}{\mathcal{M}}
\newcommand{\Pic}{\mathrm{Pic}}

\newcommand{\orb}{\text{orb}}
\newcommand{\CR}{\text{CR}}
\newcommand{\enef}{\widetilde{\text{Nef}}}
\newcommand{\one}{\textbf{1}}

\newcommand{\Lat}{\mathbb{L}}
\newcommand{\Curve}{\mathcal{C}}
\newcommand{\diag}{\mathrm{diag}}

\DeclareMathOperator{\Spec}{Spec}

\DeclareMathOperator{\Hom}{Hom}

\newcommand{\inner}[1]{\langle #1 \rangle}

\usepackage{tikz}
\usepackage[noadjust]{cite}

\begin{document}
\title{The Open Crepant Transformation Conjecture for Toric Calabi-Yau 3-Orbifolds}
\author{Song Yu}
\address{Yau Mathematical Sciences Center, Tsinghua University, Beijing, China}
\email{song-yu@tsinghua.edu.cn}

\begin{abstract}
We prove an open version of Ruan's Crepant Transformation Conjecture for toric Calabi-Yau 3-orbifolds, which is an identification of disk invariants of $K$-equivalent semi-projective toric Calabi-Yau 3-orbifolds relative to corresponding Lagrangian suborbifolds of Aganagic-Vafa type. Our main tool is a mirror theorem of Fang-Liu-Tseng that relates these disk invariants to local coordinates on the B-model mirror curves. Treating toric crepant transformations as wall-crossings in the GKZ secondary fan, we establish the identification of disk invariants through constructing a global family of mirror curves over charts of the secondary variety and understanding analytic continuation on local coordinates. Our work generalizes previous results of Brini-Cavalieri-Ross on disk invariants of threefold type-A singularities and of Ke-Zhou on crepant resolutions with effective outer branes.
\end{abstract}

\maketitle
%

\section{Introduction}

\subsection{The Crepant Transformation Conjecture}
The \emph{Crepant Transformation Conjecture} of Ruan \cite{Ruan02, Ruan06} asserts that a pair of $K$-equivalent \cite{Wang98} manifolds or Gorenstein orbifolds $\X_\pm$ have isomorphic quantum cohomologies. Bryan-Graber \cite{BG09} formulated this conjecture in terms of Gromov-Witten theory, namely that there is an identification of genus-zero Gromov-Witten potentials $F_0^{\X_\pm}$ under a graded linear isomorphism between Chen-Ruan cohomologies $H_\CR^*(\X_\pm;\C)$ and analytic continuation in the quantum variables. In the Givental formalism of the conjecture \cite{Givental04, CIT09, CR13}, the genus-zero Gromov-Witten invariants of $\X_\pm$ are encoded in a Lagrangian cone $\L_\pm$ in a symplectic vector space $\H_\pm$, and the assertion is the existence of a graded linear symplectomorphism between $\H_\pm$ that identifies $\L_\pm$ under analytic continuation. Extended statements that include higher-genus Gromov-Witten invariants have also been formulated. The Crepant Transformation Conjecture is closely related to the study of quantum cohomology and Gromov-Witten invariants using mirror symmetry. In a typical scenario, the $K$-equivalent pair $\X_\pm$ correspond to distinguished points $P_\pm$ in a B-model moduli space, and the Gromov-Witten potentials of $\X_\pm$ are mirror to local solutions to a global system of differential equations near $P_\pm$. The desired identification of Gromov-Witten theories of $\X_\pm$ is expected to result from analytic continuation of such local solutions on the B-model moduli.

The Crepant Transformation Conjecture has been verified to various extents of generalities and has become a guiding principle in the study of the relation between quantum cohomology and birational geometry. See for instance \cite{BMP09, BMP11, BGh09, BGh092, BGP08, CLLZ14, CLZZ09, Coates09, CCIT09, CI18, CIJ18, CIT09, Gillam13, GW12, ILLW12, LLQW16, LLW10, LLW12, LLW16, LP18, McLean18, Perroni07, Wise11}. It has received particular success in the toric setting, where most required ingredients have explicit descriptions. The birational geometry of toric orbifolds (or smooth toric Deligne-Mumford stacks \cite{FMN10, BCS05} with trivial generic stablizer) is combinatorially understood in terms of wall-crossing in the GKZ secondary fan or variation of GIT stability conditions \cite{CLS11, DH98, FJR18, Thaddeus96}. From the Stanley-Reisner presentation of Chen-Ruan cohomology \cite{BCS05}, there is a natural graded linear isomorphism between $H_\CR^*(\X_\pm;\C)$ for a $K$-equivalent pair $\X_\pm$. Finally, a mirror theorem for general semi-projective toric orbifolds is known \cite{CCIT15}, and the secondary variety naturally arises as the B-model moduli for analytic continuation. As examples of previous progress, Coates-Corti-Iritani-Tseng \cite{CCIT09} proved the genus-zero Crepant Transformation Conjecture for the $A_n$-surface and its full resolution in the Bryan-Graber formalism, and Coates-Iritani-Jiang \cite{CIJ18} for general $K$-equivalent semi-projective toric orbifolds or complete intersections in the Givental formalism. The higher-genus Crepant Transformation Conjecture was proven for the $A_n$-surface and its resolution by Zhou \cite{Zhou08}, and for general compact weak-Fano toric orbifolds by Coates-Iritani \cite{CI18}.

More recently, there has also been much progress on the behavior of quantum cohomology under \emph{discrepant} birational transformations. We refer to \cite{AS18,Bayer04,CLS22,Iritani20,Iritani23,LLW17,LLW21,MX23} and the references therein.

\subsection{The Open Crepant Transformation Conjecture}
The present work focuses on the Crepant Transformation Conjecture in the extended context of open Gromov-Witten theory, which studies stable maps from bordered orbifold Riemann surfaces to orbifolds with Lagrangian boundary conditions \cite{KL01, Ross14}. We restrict our attention to toric Calabi-Yau 3-orbifolds with Lagrangian suborbifolds of Aganagic-Vafa type \cite{AV00, AKV02}, for which open Gromov-Witten invariants can be defined and computed via localization \cite{GZ02, FL13, FLT19}. Following the Bryan-Graber formalism, informally, the \emph{Open Crepant Transformation Conjecture} in this setting asserts that for a pair of $K$-equivalent toric Calabi-Yau 3-orbifolds with corresponding Aganagic-Vafa branes, there is an identification between \emph{disk potentials}, which encode stable maps from genus-zero domains with one boundary component, under analytic continuation in the moduli parameters. In general, one expects an identification of all-genus open-closed Gromov-Witten potentials, which accounts for stable maps from domains with possibly higher genus and multiple boundary components. We restrict our attention in the present work to the most fundamental case of disk potentials.

Works of Brini, Cavalieri, and Ross \cite{BC18, BCR17, CR12} provided first evidences of the Open Crepant Transformation Conjecture: Using their notion of \emph{winding neutral} disk potentials, they proved the conjecture in the Givental formalism for the examples of $[\C^2/\Z_n] \times \C$, $[\C^3/\Z_2 \times \Z_2]$, and $K_{\P(n,1,1)}$ (and their full resolutions). For the first two examples, they were further able to prove the conjecture in all genera by quantizing the winding neutral disk potentials. With a different method, Ke-Zhou \cite{KZ15} proved the conjecture for a partial resolution $\X_+ \to \X_-$ where the Aganagic-Vafa brane in $\X_\pm$ specifying the boundary condition is \emph{effective} and \emph{outer}, i.e. has trivial generic stablizer and intersects a noncompact torus-invariant curve in $\X_\pm$. They identified the disk potentials of $\X_\pm$ by using an explicit formula from Fang-Liu-Tseng \cite{FLT19}.

In the present work, we prove the Open Crepant Transformation Conjecture for semi-projective toric Calabi-Yau 3-orbifolds in full generality. This includes cases where $\X_\pm$ are related by a \emph{flop} or a \emph{partial} resolution, and where the Aganagic-Vafa branes are \emph{inner}, i.e. intersecting compact torus-invariant curves, or \emph{ineffective}, i.e. having nontrivial generic stablizers. Moreover, using mirror symmetry for disk invariants, we offer a simple formalism for phrasing and studying the conjecture, which develops the idea hinted by \cite{FLT19, FLZ19} and illustrated for the example of $K_{\P^2}$ by Fang \cite{Fang19}.

We note that the conjecture has also been studied for disk invariants of compact toric orbifolds with boundary conditions specified by a Lagrangian torus moment map fiber \cite{CCLT14}, from the perspective of SYZ mirror symmetry. For toric Calabi-Yau manifolds, disk invariants specified by the two types of boundary conditions are closely related \cite{HKLZ19}.

\subsection{Our main results}
We now state our main results more carefully. Let $\X$ be a semi-projective toric Calabi-Yau 3-orbifold, $\L$ be an Aganagic-Vafa brane in $\X$, and $f \in \Z$ be an additional parameter called the \emph{framing} of $\L$. The generic stablizer group of $\L$ is isomorphic to $\mu_{\ell}$ for some $\ell \in \Z_{>0}$ and $H_1(\L;\Z) \cong \Z \times \mu_\ell$. Let $F^{\X, (\L,f)}= F_{0,1}^{\X, (\L,f)}$ be the \emph{A-model disk potential} of $(\X, \L, f)$ defined by \cite{FLT19}. It takes value in the Chen-Ruan cohomology $H_\CR^*(\B \mu_\ell; \C)$, and depends on closed moduli parameters $Q = (Q_1, \dots, Q_{\dim H^2_{\CR}(\X;\C)})$ and an \emph{open} moduli parameter $X$ that parametrizes the map from the boundary component. (Our variable $Q_a$ is $\tau_a$ in \cite{FLT19} and our $X$ is $Q^bX_1$ in \cite{FLT19}.) For instance, when $\L$ is an outer brane, we have
\[ F^{\X, (\L,f)}(Q,X) = \sum_{\substack{\beta' \in H_2(X, L;\Z) \text{ effective}\\ n \ge 0 \\ (d, \lambda) \in H_1(\L;\Z)}} \frac{\inner{(\tau_2(Q))^n}^{\X, (\L,f)}_{0, \beta', (d, \lambda)}}{n!}X^d\xi_\ell^{\ell-\bar{\lambda}} \one_{\bar{\lambda}},   \]
where 
\begin{itemize}
    \item $X,L$ denote the coarse moduli space of $\X, \L$ respectively;

    \item $\tau_2(Q)$ is a certain equivariant second Chen-Ruan cohomology class of $\X$;

    \item $\inner{}^{\X, (\L,f)}_{0, \beta', (d, \lambda)}$ is the disk invariant that virtually counts degree $\beta'$ maps with boundary winding number $d$ and monodromy $\lambda$;

    \item $\{\one_0, \dots, \one_{\ell-1}\}$ is the homogenous basis of $H_\CR^*(\B \mu_\ell; \C)$ consisting of identity elements of the inertia components;

    \item $\xi_\ell = \exp(-\frac{\pi \sqrt{-1}}{\ell})$ and $\bar{\lambda} \in \{0, \dots, \ell-1\}$ is specified by $\lambda = \exp(\frac{2\pi\bar{\lambda}\sqrt{-1}}{\ell})$.
\end{itemize}
The definition of $F^{\X, (\L,f)}$ in the case where $\L$ is inner as well as additional discussions will be given in Section \ref{sect:DiskInvariants}. There is a decomposition
\[  F^{\X, (\L,f)} = F_{\ell}^{\X, (\L,f)}\one_0 + \sum_{j = 1}^{\ell-1} F_j^{\X, (\L,f)} \xi_\ell^{\ell-j}\one_j,  \]
where each component $F_j^{\X, (\L,f)}$ is a series with rational coefficients.

As a basic step, we consider a pair of semi-projective toric Calabi-Yau 3-orbifolds $\X_\pm$ that differ by a single wall-crossing in the secondary fan. Either $\X_\pm$ are related by a flop, or $\X_+$ is a partial resolution of $\X_-$. Let $(\L_-, f_-)$ be a framed Aganagic-Vafa brane in $\X_-$ whose generic stablizer group has order $\ell$. In the case $\L_-$ is disjoint from the exceptional locus of the crepant transformation, it corresponds to an Aganagic-Vafa brane $\L_+$ in $\X_+$ with generic stablizer group of order $\ell$ and framing $f_+$ depending on $f_-$.

\begin{theorem}[Brane preserved] \label{thm:Main1}
In the situation above, there is a component-wise identification of disk potentials
\[  F^{\X_+, (\L_+,f_+)} = F^{\X_-, (\L_-, f_-)}  \]
under analytic continuation in the open-closed moduli parameters and framing relations.
\end{theorem}

In the case where $\L_-$ is ineffective and the partial resolution $\X_+ \to \X_-$ partially resolves the singularity along $\L_-$, there are \emph{two} Aganagic-Vafa branes $\L_+^1, \L_+^2$ in the preimage of $\L_-$ with framings $f_+^1, f_+^2$ depending on $f_-$. The orders $\ell_1, \ell_2$ of the generic stablizer groups of $\L_+^1, \L_+^2$ add up to $\ell$.

\begin{theorem}[Resolution along ineffective brane] \label{thm:Main2}
In the situation above, there is an identification
\[  \begin{bmatrix}
      F_1^{\X_+, (\L_+^1, f_+^1)}\\
      \vdots\\
      F_{\ell_1}^{\X_+, (\L_+^1, f_+^1)}\\
      F_1^{\X_+, (\L_+^2, f_+^2)}\\
      \vdots\\
      F_{\ell_2}^{\X_+, (\L_+^2, f_+^2)}
    \end{bmatrix} =
    U \begin{bmatrix}
        F_1^{\X_-, (\L_-, f_-)}\\
        \vdots\\
        F_\ell^{\X_-, (\L_-, f_-)}
      \end{bmatrix} \]
under analytic continuation in the open-closed moduli parameters and framing relations, where $U = U(\ell_1, \ell_2) \in GL(\ell;\C)$ is an invertible $\ell$-by-$\ell$ matrix depending on $\ell_1, \ell_2$ only.
\end{theorem}

We provide a more precise statement of these results in Theorem \ref{thm:OCTC}, which is stated in terms of B-model disk potentials $W^{\X,(\L,f)}$ pulled back from $F^{\X, (\L,f)}$ under the open-closed mirror map. Since a general pair of $K$-equivalent semi-projective toric Calabi-Yau 3-orbifolds are related by a sequence of toric wall-crossings, the Open Crepant Transformation Conjecture is implied by Theorems \ref{thm:Main1} and \ref{thm:Main2} above. We are thus able to recover the results of \cite{BC18, BCR17, CR12, KZ15} (on disk invariants).

\subsection{Our method via mirror symmetry for disks}
We approach the Open Crepant Transformation Conjecture via mirror symmetry for disk invariants. For a smooth toric Calabi-Yau 3-fold $\X$, Aganagic-Vafa and Aganagic-Klemm-Vafa \cite{AV00, AKV02} constructed the B-model \emph{Hori-Vafa mirror} to $\X$, which is a family of non-compact Calabi-Yau 3-folds equipped with a superpotential obtained as a period integral of the holomorphic volume form. An Aganagic-Vafa brane $\L$ in $\X$ is then mirror to a family of 2-cycles in the Hori-Vafa mirror. They computed the superpotential by solving the equation for the \emph{mirror curve}, which is an affine curve in $(\C^*)^2$ and can viewed as a B-model equivalent to the Hori-Vafa mirror through genus-zero dimensional reduction (see \cite[Section 4]{FLZ19} and the references therein). In addition, they conjectured that the superpotential agrees with the disk potential $F^{\X, (\L, f)}$ up to a mirror transform, which can be interpreted as an agreement between $F^{\X, (\L, f)}$ and a local integral of a 1-form on the mirror curve.

The above mirror conjecture \cite{AV00, AKV02} was first studied by Graber-Zaslow \cite{GZ02} for $K_{\P^2}$ and proven by Fang-Liu \cite{FL13}. Fang-Liu-Tseng \cite{FLT19} generalized the notion of Aganagic-Vafa branes and this mirror theorem for disks to the orbifold setting. The mirror curve can also serve as the B-model in studying the mirror symmetry of higher-genus open-closed Gromov-Witten invariants. The Remodeling Conjecture of Bouchard-Klemm-Mari\~no-Pasquetti \cite{BKMP09, BKMP10} asserts that the all-genus open-closed Gromov-Witten potentials of $(\X, \L)$ can be identified with integrals of the forms obtained from the Eynard-Orantin topological recursion \cite{EO07} with the mirror curve as the input. This is proven by Eynard-Orantin \cite{EO15} in the smooth case and by Fang-Liu-Zong \cite{FLZ19} in the orbifold case.

The mirror theorem for disks \cite{FLT19} is central to our formulation and proof of the Open Crepant Transformation Conjecture. The key observation is that $K$-equivalent toric Calabi-Yau 3-orbifolds have isomorphic mirror curves for generic choice of parameters. Thus one may construct a global family of mirror curves over the B-model moduli space obtained from charts of the secondary variety. Fang \cite{Fang19} detailed this idea for the example of $K_{\P^2} \to [\C^3/\Z_3]$. For a pair $\X_\pm$ of toric Calabi-Yau 3-orbifolds that differ by a single wall-crossing in the secondary fan, we construct an \emph{open} B-model moduli space that extends the construction of \cite{CIJ18} in the closed setting by incorporating the open moduli parameter. The defining equations of the mirror curves of $\X_\pm$ then fit into a global family of equations over the open B-model moduli, and by the disk mirror theorem, the disk potentials of $\X_\pm$ are mirror to local solutions to the global mirror curve equation near distinguished points. The desired identification of disk potentials is then achieved via analytic continuation on these local solutions.


\subsection{Organization of the paper}
We start in Section \ref{sect:setup} with a review of the geometry and combinatorics of toric Calabi-Yau 3-orbifolds and Aganagic-Vafa branes in them. Moreover, we recall the definitions of the disk invariants and potentials, given by \cite{FLT19}. In Section \ref{sect:CrepantDisk}, we discuss the different cases of toric wall-crossings in detail and precisely formulate the Open Crepant Transformation Conjecture (Theorem \ref{thm:OCTC}), to be proven in the remainder of the paper. As the first step, in Section \ref{sect:MCMSDisk}, we define mirror curves and study the structure of mirror curve equations in detail. In particular, we show that $K$-equivalent toric Calabi-Yau 3-orbifolds have isomorphic mirror curves for generic choices of parameters. This sets ground for the construction of the open B-model moduli space and the global mirror curve equation in Section \ref{sect:GlobalMirrorCurve}. We then finish the proof with a careful analysis of analytic continuation on local solutions and monodromies.

\subsection{Acknowledgments}
I would like to thank Chiu-Chu Melissa Liu for suggesting this project and offering unreserved support all along. I also thank Hiroshi Iritani for explaining the monodromy argument in \cite{CCIT09}, which inspires our treatment of monodromies of solutions to the mirror curve equation, and Bohan Fang and Zhengyu Zong for helpful discussions and suggestions. Finally, I thank the anonymous referees for carefully reviewing this work and providing helpful feedback.

\section{Disk Invariants of Toric Calabi-Yau 3-Orbifolds}\label{sect:setup}
In this section, we introduce the geometric setup and the disk invariants, mostly following \cite{FLT19}. Along the way, we also introduce notations to be used throughout. We work over $\C$.

\subsection{Toric Calabi-Yau 3-orbifolds}\label{sect:ToricDMStacks}
Let $\X$ be a 3-dimensional \emph{toric orbifold}, i.e. a 3-dimensional smooth toric Deligne-Mumford stack \cite{BCS05, FMN10} whose generic stablizer is trivial. Combinatorially, $\X$ is specified by a triple $(N, \Sigma(\X), \beta)$, where $N \cong \Z^3$, $\Sigma = \Sigma(\X)$ is a simplicial fan in $N_\R = N \otimes \R$, and $\beta$ is a group homomorphism
\[  \beta:\tilde{N} = \bigoplus_{i = 1}^R \Z \tilde{b}_i \to N, \qquad \tilde{b}_i \mapsto b_i \]
that satisfies:
\begin{itemize}
\item The set of 1-dimensional cones in $\Sigma$ is $\{\rho_i \mid i \in I_K\}$ for some $I_K = I_K(\X) \subseteq \{1, \dots, R\}$, where $\rho_i = \R_{\ge 0} b_i$.
\item The vectors in $\{b_i \mid i \in I_K\}$ generate a subgroup of finite index in $N$.
\item The vectors $b_1, \dots, b_R$ are distinct, belong to $|\Sigma|$, and together generate $N$ over $\Z$. In particular, $\beta$ is surjective.
\end{itemize}
The triple $(N,\Sigma, \beta)$ is an extended stacky fan (in the sense of \cite{Jiang08}) whose underlying stacky fan (in the sense of \cite{BCS05}) is $(N, \Sigma, \beta|_{\bigoplus_{i \in I_K} \Z\tilde{b}_i})$.
There is a short exact sequence of free abelian groups:
\begin{equation}\label{eq:StackyFanSq}
\xymatrix{
0 \ar[r] & \Lat \ar[r]^\psi & \tilde{N} \ar[r]^\beta & N \ar[r] & 0,
}
\end{equation}
where $\Lat$ has rank $k = R-3$. We set $I_\orb = I_\orb(\X) = \{1, \dots, R\} \setminus I_K$.

Applying $-\otimes \C^*$ to the sequence (\ref{eq:StackyFanSq}), we obtain a short exact sequence of complex tori:
\begin{equation}\label{eq:ToriSq}
\xymatrix{
0 \ar[r] & G \ar[r] & \tilde{T} \ar[r] & T \ar[r] & 0.
}
\end{equation}
$\tilde{T}$ naturally acts on $\C^R = \tilde{N} \otimes \C$, and the inclusion $G \to \tilde{T}$ induces a $G$-action on $\tilde{T}$ that extends to $\C^R$. Then $\X$ is geometrically represented as a quotient stack $[U/G]$ for some dense open subset $U \subseteq \C^R$ (see \cite{Jiang08}). The DM torus of $\X$ is the quotient stack $\T=[\tilde{T}/G]$, and the coarse moduli space $X$ of $\X$ is the simplicial toric 3-fold defined by the fan $\Sigma$, which contains the torus $T$.

\begin{assumption}\label{assump:Assump} \rm{
In this paper, we make the following assumptions on $\X$:
\begin{itemize}
\item $\X$ is \emph{Calabi-Yau}: There is a basis $\{v_1, v_2, v_3\}$ of $N$ with dual basis $\{u_1, u_2, u_3\}$ of $M = \Hom(N, \Z)$ such that $\inner{b_i, u_3} = 1$ for all $i = 1, \dots, R$, where $\inner{-,-}:N \times M \to \Z$ is the natural pairing extended linearly over $\R$.
\item The coarse moduli space $X$ is \emph{semi-projective}: $X$ is projective over $\Spec H^0(X, \O_X)$.
\end{itemize}
}\end{assumption}
By \cite[Proposition 14.4.1]{CLS11}, $X$ is semi-projective if and only if $|\Sigma| = \sum_{i = 1}^R \R_{\ge 0} b_i$. Thus our assumptions imply that all maximal cones in $\Sigma$ are 3-dimensional. The Calabi-Yau condition implies that all the $b_i$'s lie on the hyperplane $N' = \{ v \in N_\R \mid \inner{v, u_3} = 1\}$. Then $\Delta = |\Sigma| \cap N'$ is a convex lattice polyhedron and $b_1, \dots, b_R$ is a complete list of lattice points in $\Delta$. Note that $\Sigma$ induces a plane graph supported on $\Delta$, where the intersection of $\Delta$ and the support of the 1-, 2-, and 3-dimensional cones give rise to the vertices, edges, and bounded faces of the graph respectively. This graph induces a regular triangulation of $\Delta$ in the sense of \cite[Section 15.2]{CLS11}. The dual plane graph is called the toric graph of $\X$.

\begin{example}[$A_n$ singularities] \label{ex:An} \rm{
For $n \in \Z_{>0}$, $\A_n = [\C^2/\Z_{n+1}] \times \C$ is a toric Calabi-Yau 3-orbifold. Its fan consists of one 3-dimensional cone $\sigma$ spanned by
\[  (1,0,1), \qquad (0, n+1, 1), \qquad (0, 0, 1), \]
which lie on the affine hyperplane $N'$ consisting of vectors whose last coordinate is $1$. See Figure \ref{fig:AnOuter} for the induced polygon and toric graph. To obtain the extended stacky fan, we choose extra vectors $(0,i, 1)$ for $i = 1, \dots, n$.
\begin{figure}[htb]
  \[
      \begin{tikzpicture}[scale=.8]
          \node[label=right:{$(1,0,1)$}] at (1,0) {$\bullet$};
          \node[label=left:{$(0,0,1)$}] at (0,0) {$\bullet$};
          \node at (0,1) {$\circ$};
          \node at (0,2.5) {$\circ$};
          \node[label=left:{$(0,n+1,1)$}] at (0,3.5) {$\bullet$};
          \node[label=left:{$\tau$}] at (0, 1.75) {};
          \node at (0.4,0.5) {$\sigma$};
          \draw (1,0) -- (0,0) -- (0,3.5) -- (1,0);

          \node at (7,2) {$\bullet$};
          \node[label=right:{$\V(\sigma)$}] at (7,1.8) {};
          \draw (7, 0) -- (7,2) -- (10,3.5);
          \draw[decorate, decoration=snake] (7,2) -- (4,2);
          \node[label=above:{$\V(\tau)$}] at (5.5,2) {};
      \end{tikzpicture}
  \]
  \caption{The stacky fan and the toric graph of $\A_n$. There is one torus fixed point $\V(\sigma)$ and three torus invariant curves. The curve $\V(\tau)$ has nontrivial generic stablizer group.}
  \label{fig:AnOuter}
\end{figure}
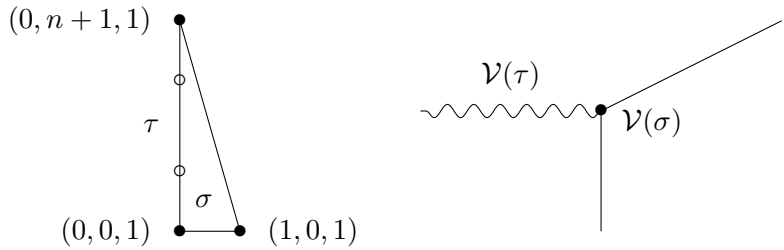

Moreover, $\A_n' = [\O_{\P^1}(-1) \oplus \O_{\P^1}(-1)/\Z_{n+1}]$ is a toric Calabi-Yau 3-orbifold, which can be formed by gluing together two copies of $\A_n$. The map $\beta$ in the extended stacky fan is given by the matrix
\[   \begin{bmatrix}
        1 & 0 & 0 & -1 & 0 & 0 & \cdots & 0\\
        0 & n+1 & 0 & 0 & 1 & 2 & \cdots & n\\
        1 & 1 & 1 & 1 & 1 & 1 & \cdots & 1
      \end{bmatrix}.  \]
See Figure \ref{fig:AnInner} for the induced polygon and toric graph.
\begin{figure}[htb]
  \[
      \begin{tikzpicture}[scale=.8]
          \node[label=right:{$b_1$}] at (1,0) {$\bullet$};
          \node[label=left:{$b_4 = (-1,0,1)$}] at (-1,0) {$\bullet$};
          \node[label=below:{$b_3$}] at (0,0) {$\bullet$};
          \node at (0,1) {$\circ$};
          \node at (0,2.5) {$\circ$};
          \node[label=left:{$b_2$}] at (0,3.5) {$\bullet$};
          \node at (-0.2, 1.75) {$\tau$};
          \node at (0.4,0.5) {$\sigma$};
          \node at (-0.4,0.5) {$\sigma'$};
          \draw (1,0) -- (0,0) -- (0,3.5) -- (1,0);
          \draw (0,0) -- (-1,0) -- (0, 3.5);

          \node at (8,2) {$\bullet$};
          \node[label=right:{$\V(\sigma)$}] at (8,1.8) {};
          \node at (6,2) {$\bullet$};
          \node[label=left:{$\V(\sigma')$}] at (6,1.8) {};
          \draw (8, 0) -- (8,2) -- (11,3.5);
          \draw (6, 0) -- (6,2) -- (3,3.5);
          \draw[decorate, decoration=snake] (8,2) -- (6,2);
          \node[label=above:{$\V(\tau)$}] at (7,2) {};
      \end{tikzpicture}
  \]
  \caption{The stacky fan and the toric graph of $\A_n'$.}
  \label{fig:AnInner}
\end{figure}
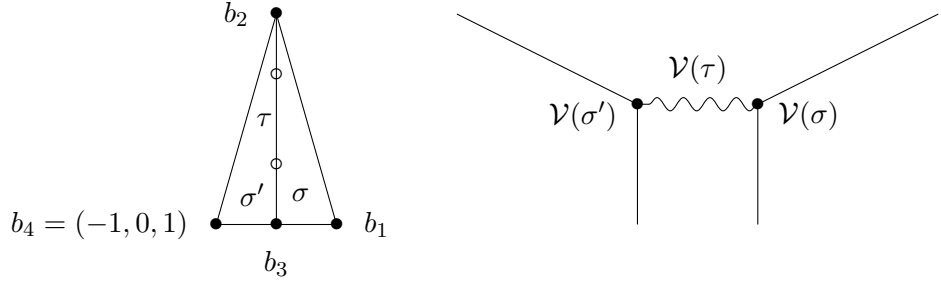
}\end{example}

\begin{example} \label{ex:KP2} \rm{
$K_{\P^2}$ is a crepant resolution of the toric Calabi-Yau 3-orbifold $[\C^3/\Z_3]$. See Figure \ref{fig:KP2} for the induced polygon triangulations.
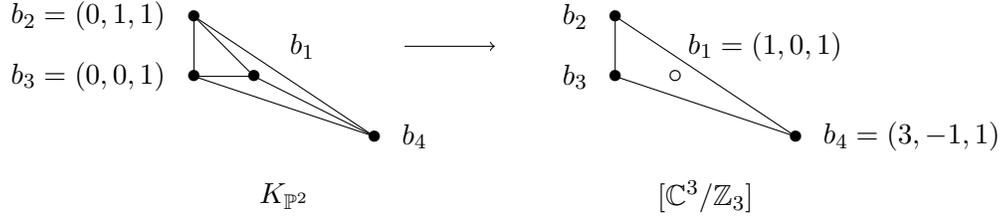
\begin{figure}[htb]
  \[
      \begin{tikzpicture}[scale=.8]
          \node[label=right:{$b_4$}] at (3,-1) {$\bullet$};
          \node[label=left:{$b_3 = (0,0,1)$}] at (0,0) {$\bullet$};
          \node[label=left:{$b_2 = (0,1,1)$}] at (0,1) {$\bullet$};
          \node at (1,0) {$\bullet$};
          \node at (1.8,0.5) {$b_1$};
          \node at (1.5, -2) {$K_{\P^2}$};
          \draw (1,0) -- (0,0) -- (0,1) -- (1,0) -- (3,-1);
          \draw (0,0) -- (3,-1) -- (0,1);

          \draw[->] (3.5,0.5) -- (5, 0.5);

          \node[label=right:{$b_4 = (3,-1,1)$}] at (10,-1) {$\bullet$};
          \node[label=left:{$b_3$}] at (7,0) {$\bullet$};
          \node[label=left:{$b_2$}] at (7,1) {$\bullet$};
          \node at (8,0) {$\circ$};
          \node at (9.5,0.5) {$b_1 = (1,0,1)$};
          \node at (8.5, -2) {$[\C^3/\Z_3]$};
          \draw (7,0) -- (10,-1) -- (7,1) -- (7,0);
      \end{tikzpicture}
  \]
  \caption{The stacky fans of $K_{\P^2}$ and $[\C^3/\Z_3]$.}
  \label{fig:KP2}
\end{figure}
}\end{example}

\subsection{Torus orbits, stablizers, and flags}\label{sect:Orbits}
For each $d = 1, 2, 3$, let $\Sigma(d)$ denote the set of cones in $\Sigma$ with dimension $i$. Then $\Sigma(1) = \{\rho_i \mid i \in I_K\}$. Let $\sigma \in \Sigma(d)$. Denote $\V(\sigma)$ as the $(3-d)$-dimensional $\T$-invariant closed substack of $\X$ and $G_\sigma \le G$ as the generic stablizer group of $\V(\sigma)$, which is a finite group. Moreover, denote
\[  I_\sigma := \{i \in I_K \mid \rho_i \not \subset \sigma\} \cup I_\orb, \qquad I_\sigma' := \{i \in I_K \mid \rho_i \subset \sigma\}. \]
Then $|I_\sigma'| = d$ and $|I_\sigma| = R-d$.

A \emph{flag} of $\Sigma$ is a pair $(\tau, \sigma) \in \Sigma(2) \times \Sigma(3)$ with $\tau \subset \sigma$. Let $F(\Sigma)$ denote the set of all flags of $\Sigma$. Given a flag $(\tau, \sigma)$, $\V(\sigma)$ is a $\T$-fixed point contained in the $\T$-invariant line $\V(\tau)$. In the toric graph, $\V(\sigma)$ corresponds to a vertex and $\V(\tau)$ corresponds to an incident edge. Moreover, $G_\tau$ is a cyclic subgroup of $G_\sigma$. Denote
\[  \ell(\tau, \sigma) := |G_\tau|, \qquad r(\tau, \sigma) := [G_\sigma: G_\tau]. \]
There is a short exact sequence of finite groups
\[
\xymatrix{
0 \ar[r] & G_\tau \ar[r] & G_\sigma \ar[r] & \mu_{r(\tau,\sigma)} \ar[r] & 0,
}
\]
where $\mu_r$ denotes the cyclic group of $r$-th roots of unities. We have $\V(\sigma) \cong \B G_\sigma$. Moreover, the coarse moduli space of $\V(\tau)$ is either $\C$, in the case $\sigma$ is the unique 3-dimensional cone containing $\tau$, or $\P^1$, in the case $\tau$ is contained in a 3-dimensional cone other than $\sigma$.


\begin{example} \label{ex:AnStab} \rm{
For $\A_n$ or $\A_n'$ as in Example \ref{ex:An}, we have $G_\sigma \cong G_\tau \cong N/\inner{b_1, b_2, b_3} \cong \Z_{n+1}$. In the case of $\A_n$ (Figure \ref{fig:AnOuter}), $\V(\tau) \cong \B\mu_{n+1} \times \C$; in the case of $\A_n'$ (Figure \ref{fig:AnInner}), $\V(\tau) \cong \B\mu_{n+1} \times \P^1$.
}\end{example}

To a flag $(\tau, \sigma) \in F(\Sigma)$, we associate a basis $\{v_1(\tau, \sigma), v_2(\tau, \sigma), v_3(\tau, \sigma)\}$ of $N$ as follows: Suppose $I_\sigma' = \{i_1(\tau, \sigma),i_2(\tau, \sigma),i_3(\tau, \sigma)\}$ and $I_\tau' = \{i_2(\tau, \sigma),i_3(\tau, \sigma)\}$, where $b_{i_1(\tau, \sigma)}, b_{i_2(\tau, \sigma)}, b_{i_3(\tau, \sigma)}$ appear in counterclockwise order on the hyperplane $N'$. Then we take $\{v_1(\tau, \sigma), v_2(\tau,\sigma), v_3(\tau, \sigma)\}$ to be the unique basis such that
\[ \begin{cases}
      b_{i_1(\tau, \sigma)} = r(\tau,\sigma) v_1(\tau, \sigma) -s(\tau, \sigma) v_2(\tau, \sigma) + v_3(\tau, \sigma),\\
      b_{i_2(\tau, \sigma)} = \ell(\tau, \sigma)v_2(\tau, \sigma) + v_3(\tau, \sigma),\\
      b_{i_3(\tau, \sigma)} = v_3(\tau, \sigma),
    \end{cases}
 \]
for some $s(\tau, \sigma) \in \{0, \dots, r(\tau, \sigma)-1\}$. For each $i= 1, \dots, R$, let $m_i(\tau, \sigma), n_i(\tau, \sigma) \in \Z$ such that
\[  b_i = m_i(\tau, \sigma) v_1(\tau, \sigma) + n_i(\tau, \sigma) v_2(\tau, \sigma) + v_3(\tau, \sigma). \]
Finally, let $\{u_1(\tau, \sigma), u_2(\tau,\sigma), u_3(\tau, \sigma)\}$ be the basis of $M$ dual to $\{v_1(\tau, \sigma), v_2(\tau, \sigma), v_3(\tau, \sigma)\}$.

Given two flags in $F(\Sigma)$, the change-of-basis matrix between the two associated bases of $N$ has form
\[  \begin{bmatrix}
      \star & \star & \star\\
      \star & \star & \star\\
      0 & 0 & 1
    \end{bmatrix} \in \rm{SL}(3;\Z).
\]
In particular, the vector $u_3 \in M$ does not vary as the flag changes. One checks that $N' = \{ v \in N_R \mid \inner{v, u_3(\tau, \sigma)} = 1\}$ for any flag $(\tau, \sigma)$.

\subsection{Extended Nef cone}\label{sect:NefCone}
Applying $\Hom(-, \Z)$ to the sequence (\ref{eq:StackyFanSq}), we obtain a short exact sequence of free abelian groups:
\begin{equation}\label{eq:DualStackyFanSq}
\xymatrix{
0 \ar[r] & M \ar[r]^{\beta^\vee} & \tilde{M} \ar[r]^{\psi^\vee} & \Lat^\vee \ar[r] & 0.
}
\end{equation}
Note that $M, \tilde{M}$, and $\Lat^\vee$ are the character lattices of $T, \tilde{T}$, and $G$ respectively. Let $\{\tilde{b}_1^\vee, \dots, \tilde{b}_R^\vee\}$ be the basis of $\tilde{M}$ dual to $\{\tilde{b}_1, \dots, \tilde{b}_R\}$. For $i = 1, \dots, R$, define
\[  D_i := \psi^\vee (\tilde{b}_i^\vee) \in \Lat^\vee.  \]
There are identifications
\begin{equation}\label{eq:LIdentification}
    \Lat^\vee \bigg/ \bigoplus_{i \in I_\orb} \Z D_i \cong \Pic(\X) \cong H^2(\X;\Z)
\end{equation}
under which for each $i \in I_K$, the image of $D_i$ in the quotient is identified with the class of the divisor $\V(\rho_i)$. Over $\R$, (\ref{eq:LIdentification}) yields splittings:
\[  \Lat^\vee_\R = \Lat^\vee \otimes \R \cong (\Pic(\X) \otimes \R) \oplus \bigoplus_{i \in I_\orb} \R D_i \cong H^2(\X; \R) \oplus \bigoplus_{i \in I_\orb} \R D_i. \]

For each $\sigma \in \Sigma(3)$, define
\[  \enef_\sigma := \sum_{i \in I_\sigma} \R_{\ge 0} D_i,   \]
which is a simplicial cone in $\Lat^\vee_\R$. The \emph{extended Nef cone} of $\X$ is defined as
\[  \enef(\X) := \bigcap_{\sigma \in \Sigma(3)} \enef_\sigma,  \]
which is a top-dimensional convex polyhedral cone in $\Lat^\vee_\R$. Observe that the cone $\sum_{i \in I_\orb} \R_{\ge 0}D_i$ is a face of each $\enef_\sigma$ and is thus a face of $\enef(\X)$. With respect to the splittings above, $\enef(\X)$ is spanned by the ordinary Nef cone of $\X$ and $\sum_{i \in I_\orb} \R_{\ge 0}D_i$.

$\enef(\X)$ is referred to as the GKZ cone of $\X$ in \cite[Section 14.4]{CLS11} since it is the cone in the \emph{secondary fan} that corresponds to $\X$. From the viewpoint of GIT, if we consider the action of $G$ on $\C^R$ defined in Section \ref{sect:ToricDMStacks} and pick a character $\theta \in \Lat^\vee$ of $G$ lying in the relative interior of $\enef(\X)$, then the open subset $U \subseteq \C^R$ in Section \ref{sect:ToricDMStacks} is the semistable locus of $\theta$, and $\X$ is the GIT quotient stack $[\C^R \sslash_\theta G]$ (see for example \cite[Section 14.3]{CLS11} and \cite{FJR18}). Choosing another stability condition in the secondary fan gives rise to a toric Calabi-Yau 3-orbifold that differs from $\X$ by a toric crepant transformation. We will return to this perspective in Section \ref{sect:GLSM}.

\begin{example}\label{ex:A1SecondaryFan} \rm{
For $\A_1$, the maps $\beta$ and $\psi$ in (\ref{eq:StackyFanSq}) are given by matrices
\[  \begin{bmatrix}
        1 & 0 & 0 & 0\\
        0 & 1 & 0 & 2\\
        1 & 1 & 1 & 1
      \end{bmatrix}, \qquad
      \begin{bmatrix}
        0 \\ -2 \\ 1 \\ 1
      \end{bmatrix}.    \]
The secondary fan is the complete fan on the line $\Lat^\vee_\R = \R$, and $\enef(\A_1)$ is the ray generated by $D_2 = -2$. The other ray, generated by $D_3 = D_4 = 1$, corresponds to the crepant resolution $K_{\P^1} \oplus \O_{\P^1}$ of $\A_1$. See Figure \ref{fig:A1SecondaryFan}. The secondary variety is the weighted projective line $\P(1,2)$. Similarly, the secondary variety of $[\C^3/\Z_3]$ and $K_{\P^2}$ is $\P(1,3)$.
\begin{figure}[htb]
  \[
      \begin{tikzpicture}[scale=.8]
          \draw[<->] (-5,0) -- (5,0);
          \node[label=above:{$0 = D_1$}] at (0,0) {$|$};
          \node at (0,0) {$\bullet$};
          \node[label=below:{$-2 = D_2$}] at (-2,0) {$\bullet$};
          \node at (1,0) {$\bullet$};
          \node[label=below:{$1 = D_3 = D_4$}] at (2,0) {};
          \draw (-2,1) -- (-3,1) -- (-3, 3) -- (-2,1);
          \draw (2,2) -- (3,1) -- (2,1) -- (2, 3) -- (3,1);
      \end{tikzpicture}
  \]
  \caption{The secondary fan and extended Nef cone of $\A_1$.}
  \label{fig:A1SecondaryFan}
\end{figure}
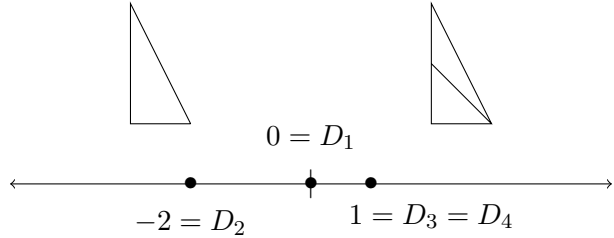
}\end{example}

\subsection{Framed Aganagic-Vafa branes}\label{sect:AVBranes}
From the Hamiltonian action of the maximal compact subgroup of $G$ on $\C^R$, $\X$ can be constructed as a symplectic quotient and endowed with a symplectic structure (see \cite[Section 2.6]{FLT19}). The boundary condition of the disk invariants we consider is a pair $(\L,f)$, where $\L$ is a Lagrangian suborbifold of $\X$ of Aganagic-Vafa type and $f \in \Z$ is the \emph{framing} of $\L$. Aganagic-Vafa branes were introduced in \cite{AV00} for smooth toric Calabi-Yau 3-folds and defined in \cite{FLT19} for toric Calabi-Yau 3-orbifolds in general.

A key property of $\L$ is that it intersects a unique 1-dimensional $\T$-orbit $\V(\tau)$ in $\X$, where $\tau \in \Sigma(2)$. Let $L$ denote the coarse moduli space of $\L$. If the coarse moduli space $V(\tau)$ of $\V(\tau)$ is $\C$, we say that $\L$ is an \emph{outer brane}. In this case, there is a unique 3-dimensional cone $\sigma \in \Sigma(3)$ that contains $\tau$, and we associate to $\L$ the flag $(\tau, \sigma)$. Moreover, $L \cap V(\tau)$ bounds a disk $D$ in $V(\tau)$. On the other hand, if $V(\tau)$ is $\P^1$, we say that $\L$ is an \emph{inner brane}. In this case, there are two 3-dimensional cones $\sigma,\sigma' \in \Sigma(3)$ that contain $\tau$, and we associate to $\L$ the flags $(\tau, \sigma), (\tau, \sigma')$. Moreover, $L \cap V(\tau)$ bounds two disks $D, D'$ centered at the $T$-fixed points $V(\sigma), V(\sigma')$ respectively. See Figure \ref{fig:AVBranes} for an illustration of Aganagic-Vafa branes in the example of $\A_1'$.

The order of the generic stablizer group of $\L$ is equal to that of $G_\tau$, which is $\ell = \ell(\tau, \sigma)$. We have $H_1(\L;\Z) \cong \Z \times G_\tau$. We say that $\L$ is \emph{effective} if $\ell = 1$, and \emph{ineffective} if $\ell>1$.

Viewing $M$ as the lattice of characters of the 3-dimensional torus $T$, we denote $T' :=\ker(u_3) \cong (\C^*)^2$ as the 2-dimensional Calabi-Yau subtorus of $T$ and $T'_\R \cong U(1)^2$ as the maximal compact subgroup of $T'$. Then $T'_\R$ acts holomorphically on $\X$ and preserves $\L$. The toric graph of $\X$ is the image of the $0$- and $1$-dimensional $\T$-orbits under the moment map given by the $T'_\R$ action. The image of $\L$ under the moment map is an interior point on the edge which is the image of $\V(\tau)$. Furthermore, the framing $f$ determines a 1-dimensional framing subtorus $T'_f := \ker(u_1 - fu_2, u_3)$ of $T'$.

\begin{figure}[htb]
  \[
      \begin{tikzpicture}[scale=.7]
          \draw (1.5,0) -- (0,0) -- (0,3) -- (1.5,0);
          \draw (0,0) -- (-1.5,0) -- (0, 3);
          \draw[ultra thick] (0.2,1.4) -- (-0.2,1.4);
          \node at (-0.4, 0.9) {$\L_1$};
          \draw[ultra thick] (0.57,1.36) -- (0.93,1.54);
          \node at (1.4, 1.6) {$\L_2$};
          \draw (1, -3) -- (1,-2) -- (2.5,-1.2);
          \draw (-1, -3) -- (-1,-2) -- (-2.5,-1.2);
          \draw[decorate, decoration=snake] (-1,-2) -- (1,-2);
          \node at (0,-2) {$\bullet$};
          \node at (0, -1.2) {$\L_1$};
          \node at (1.75, -1.6) {$\bullet$};
          \node at (1.5, -1) {$\L_2$};
      \end{tikzpicture} \qquad
      \begin{tikzpicture}[scale=.85]
          \coordinate (1) at (-3,0);
          \coordinate (2) at (3,0);
          \node at (1) {$\bullet$};
          \node at (2) {$\bullet$};
          \draw (1) to[bend right] (2);
          \draw (1) to[bend left] (2);
          \draw[thick] (0, 0.86) to[bend left] (0,-0.86);
          \draw[dashed, thick] (0, 0.86) to[bend right] (0,-0.86);
          \draw[dashed] (-0.7, 3) -- (0.7,2) -- (0.7, -3) -- (-0.7, -2) -- (-0.7, 3);
          \node at (-2.5, 0.8) {$V(\sigma')$};
          \node at (2.5, 0.8) {$V(\sigma)$};
          \node at (0, -1.4) {$V(\tau_1)$};
          \node at (-0.3, 2.1) {$L_1$};
          \node at (1.5, 0) {$D$};
          \node at (-1.5, 0) {$D'$};

          \draw (1) .. controls (-3.7, -1) and (-3.7, -2) .. (-3.7, -3);
          \draw (1) .. controls (-2.3, -1) and (-2.3, -2) .. (-2.3, -3);
          \draw (2) .. controls (3.7, -1) and (3.7, -2) .. (3.7, -3);
          \draw (2) .. controls (2.3, -1) and (2.3, -2) .. (2.3, -3);
          \draw (-3.66, -2) to[bend right] (-2.34, -2);
          \draw[dashed] (-3.66, -2) to[bend left] (-2.34, -2);
          \draw (3.66, -2) to[bend left] (2.34, -2);
          \draw[dashed] (3.66, -2) to[bend right] (2.34, -2);

          \draw (1) .. controls (-3, 1) and (-4.5,2.3) .. (-5, 2.5);
          \draw (1) .. controls (-4, 0) and (-5.5, 0.8) .. (-6, 1.2);
          \draw (2) .. controls (3, 1) and (4.5,2.3) .. (5, 2.5);
          \draw (2) .. controls (4, 0) and (5.5, 0.8) .. (6, 1.2);

          \draw (-4, 1.75) to[bend left] (-5, 0.6);
          \draw[dashed] (-4, 1.75) to[bend right] (-5, 0.6);
          \draw[thick] (4, 1.75) to[bend left] (5, 0.6);
          \draw[dashed, thick] (4, 1.75) to[bend right] (5, 0.6);
          \draw[dashed] (2.3, 2.6) -- (4,3) -- (6.6, -0.1) -- (4.9, -0.5) -- (2.3, 2.6);
          \node at (5.7, 1.9) {$V(\tau_2)$};
          \node at (3,2.4) {$L_2$};
          \node at (3.8, 0.7) {$D$};
        \end{tikzpicture}
  \]
  \caption{Aganagic-Vafa branes in $\A_1'$. $\L_1$ is an inner brane since it intersects the compact torus orbit $\V(\tau_1)$. Moreover, $\L_1$ has nontrivial generic stablizer group equal to $G_{\tau_1} \cong \Z_2$. $\L_2$ is an outer brane since it intersects the noncompact torus orbit $\V(\tau_2)$. At the level of coarse moduli, the intersections of the branes with the torus invariant curves bound disks centered at the torus fixed points.}
  \label{fig:AVBranes}
\end{figure}
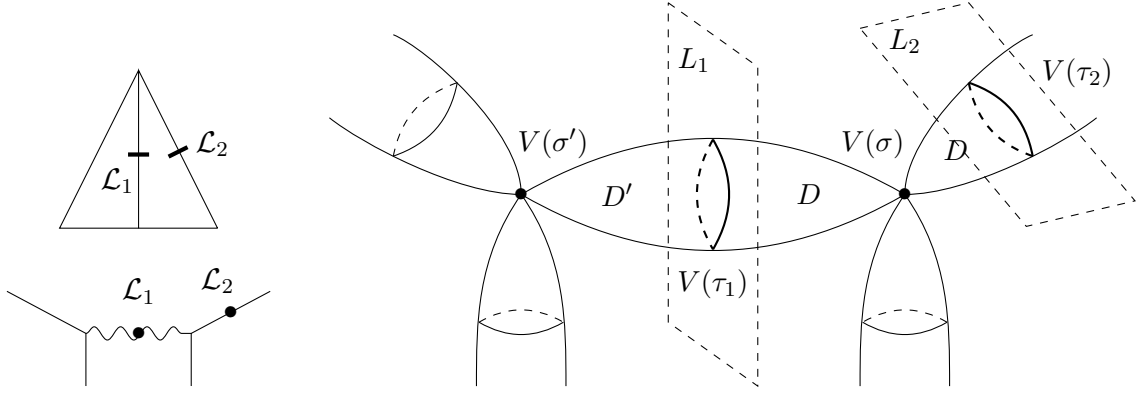

\subsection{Disk invariants and potentials}\label{sect:DiskInvariants}
Fang-Liu-Tseng \cite{FLT19} defined and computed open-closed Gromov-Witten invariants for $(\X, (\L,f))$, which are virtual counts of stable maps from bordered orbifold Riemann surfaces to $\X$ with boundary mapping to $\L$. We focus on \emph{disk invariants}, which account for maps from domains with arithmetic genus zero and one boundary component, and the generating functions that package them. We recall the necessary definitions in this section and refer the reader to \cite{FLT19} for a full treatment.


For an effective class $\beta' \in H_2(X,L;\Z)$ and a class $(d, \lambda) \in H_1(\L;\Z) \cong \Z \times G_\tau$, let
\[  \inner{}^{\X, (\L,f)}_{0, \beta', (d, \lambda)} \]
be the disk invariant of $(\X,(\L,f))$ that virtually counts degree $\beta'$ maps with boundary winding number $d$ and monodromy $\lambda$, allowing insertions at interior marked points. The disk invariants are assembled into a generating function $F^{\X, (\L, f)}=F^{\X, (\L, f)}_{0,1}$,
which we refer to as the \emph{A-model disk potential} of $(\X, (\L,f))$. It depends on \emph{closed (K\"ahler) moduli parameters} $Q = (Q_1, \dots, Q_k)$ and an \emph{open moduli parameter} $X$. $F^{\X, (\L, f)}$ takes value in the Chen-Ruan cohomology $H^*_{\CR}(\B G_\tau; \C)$. Let $(\tau, \sigma)$ be the flag associated to $\L$ as in the previous subsection. Then we have an identification
\[   G_\tau \cong N/\inner{v_1(\tau, \sigma), v_3(\tau, \sigma), b_2(\tau, \sigma)} = \{\overline{0}, \overline{v_2(\tau, \sigma)}, \dots, \overline{(\ell-1)v_2(\tau, \sigma)} \}.   \]
Let $\one_j$ be the unit of the cohomology of the inertia component of $\I \B G_\tau$ corresponding to $\overline{j v_2(\tau, \sigma)}$. Then $\{\one_0, \one_1, \cdots, \one_{\ell-1}\}$ is a homogeneous basis of $H^*_{\CR}(\B G_\tau; \C)$. 

In the case where $\L$ is an outer brane, we have
\[ F^{\X, (\L,f)}(Q,X) = \sum_{\substack{\beta \in H_2(X;\Z) \text{ effective}\\ n \ge 0 \\ (d, \lambda) \in H_1(\L;\Z), \text{ } d >0}} \frac{\inner{(\tau_2(Q))^n}^{\X, (\L,f)}_{0, \beta+d[D], (d, \lambda)}}{n!}X^d\xi_\ell^{\ell-\bar{\lambda}} \one_{\bar{\lambda}},   \]
where $\tau_2(Q)$ is a certain equivariant second Chen-Ruan cohomology class of $\X$ that depends on $Q$, $\xi_\ell = \exp(-\frac{\pi \sqrt{-1}}{\ell})$, and $\bar{\lambda} \in \{0, \dots, \ell-1\}$ is such that $\overline{\bar{\lambda}v_2(\tau,\sigma)}$ is identified with $\lambda$. $F^{\X, (\L,f)}(Q,X)$ is a series in $X\Q[[Q,X]]$.

Similarly, in the case where $\L$ is an inner brane, we have
\begin{align*}
    F^{\X, (\L,f)}(Q, q^\alpha, X) = & \sum_{\substack{\beta \in H_2(X;\Z) \text{ effective}\\ n \ge 0 \\ (d, \lambda) \in H_1(\L;\Z), \text{ } d >0}} \frac{\inner{(\tau_2(Q))^n}^{\X, (\L,f)}_{0, \beta+d[D], (d, \lambda)}}{n!}X^d\xi_\ell^{\ell-\bar{\lambda}} \one_{\bar{\lambda}} \\
    & + \sum_{\substack{\beta \in H_2(X;\Z) \text{ effective}\\ n \ge 0 \\ (d, \lambda) \in H_1(\L;\Z), \text{ } d < 0}} \frac{\inner{(\tau_2(Q))^n}^{\X, (\L,f)}_{0, \beta-d[D'], (d, \lambda)}}{n!}(q^\alpha)^{-d}X^d\xi_\ell^{\ell-\bar{\lambda}} \one_{\bar{\lambda}},
\end{align*}
where $q^\alpha$ is an additional parameter related to the symplectic area of $V(\tau)$. $F^{\X, (\L,f)}(Q, q^\alpha, X)$ is the sum of a series in $X\Q[[Q, X]]$ and a series in $q^\alpha X^{-1}\Q[[Q, q^\alpha X^{-1}]]$. We note that the quantity $q^\alpha$ is in fact a monomial in the $B$-model closed moduli parameters $q$ to be defined later. It corresponds to $Q^\alpha$ in \cite{FLT19}. See Remarks \ref{rem:QAlphaFLT} and \ref{rem:LRL} for additional discussions.

In either the outer or the inner case, we can write
\begin{equation}\label{eq:AmodelDiskPotentialDecompose}
    F^{\X, (\L, f)} = F_\ell\one_0 + \sum_{j=1}^{\ell-1} F_j \xi_\ell^{\ell-j} \one_j.
\end{equation}
This decomposition can be viewed as a grouping of disk invariants according to the monodromy of the stable map $f$ at the boundary $\partial \Curve$. Moreover, each term of each $F_j$ contains a nonzero power of $X$ as a factor.

Under the equivariant closed mirror map $Q = Q(q)$ of Coates-Corti-Iritani-Tseng \cite{CCIT15} and an open mirror map $X = X(q,x)$, $F^{\X, (\L, f)}$ pulls back to a series $W^{\X, (\L, f)}(q, x)$, which we refer to as the \emph{B-model disk potential} of $(\X, (\L,f))$. Here, $q = (q_1, \dots, q_k)$ are \emph{closed (complex) moduli parameters}
and $x$ is the \emph{B-model open moduli parameter} (see \cite[Section 4.2]{FLT19}). $W^{\X, (\L, f)}$ also takes value in the Chen-Ruan cohomology $H^*_{\CR}(\B G_\tau; \C)$, and there is a decomposition similar to (\ref{eq:AmodelDiskPotentialDecompose}):
\begin{equation}\label{eq:DiskPotentialDecompose}
    W^{\X, (\L, f)}(q, x) = W_\ell(q,x)\one_0 + \sum_{j=1}^{\ell-1} W_j(q,x) \xi_\ell^{\ell-j} \one_j.
\end{equation}
If $\L$ is outer, each $W_j$ is a series in $x\Q[[q,x]]$; if $\L$ is inner, each $W_j$ is the sum of a series in $x\Q[[q, x]]$ and a series in $x^{-1}\Q[[q, x^{-1}]]$. In particular, each term of each $W_j$ contains a nonzero power of $x$ as a factor. 


The B-model disk potential $W^{\X, (\L, f)}$ has an explicit formula (see \cite[Theorem 4.3]{FLT19}) and relates directly to local integrals on the mirror curves under the mirror theorem for disks, as we shall see in Section \ref{sect:MirrorThmDisk}. In formulating and proving the Open Crepant Transformation Conjecture in the rest of the paper, we will work with the B-model disk potential $W^{\X, (\L, f)}$ and the B-model open-closed moduli parameters $q,x$.

\section{Toric Crepant Transformations and Disk Invariants}\label{sect:CrepantDisk}
Toric crepant transformations can be viewed as wall-crossings in the GKZ secondary fan or in GIT stability conditions. In this section, we describe crepant transformations between toric Calabi-Yau 3-orbifolds and the induced transformations between Aganagic-Vafa branes. We then formulate the Open Crepant Transformation Conjecture that we will prove in subsequent sections.

\subsection{Toric wall-crossing}\label{sect:GLSM}
A toric Calabi-Yau 3-orbifold that differs from $\X$ by a toric crepant transformation is specified by an extended stacky fan $(N, \Sigma', \beta)$ with the same $N$ and $\beta$ as in that of $\X$. Such toric orbifolds are classified by their induced regular triangulations of the polygon $\Delta$, which correspond to top-dimensional cones in the \emph{secondary fan} in $\Lat^\vee_\R$ (see \cite[Proposition 15.2.9]{CLS11}). From the perspective of GIT, different cones in the secondary fan give different stability conditions as we construct these toric orbifolds as GIT quotient stacks of $\C^R$ under the action of $G$ \cite{CLS11, DH98, FJR18, Thaddeus96}.

Our main goal is to study the relations among disk invariants of these toric Calabi-Yau 3-orbifolds. We consider as the basic step a pair of toric Calabi-Yau 3-orbifolds $\X_+$ and $\X_-$ which differ by a single wall-crossing in the secondary fan, i.e. whose extended Nef cones $\enef(\X_+)$ and $\enef(\X_-)$ have a common codimension-one face $W$. Denote $\Sigma_\pm = \Sigma(\X_\pm)$. Wall-crossings of smooth toric Deligne-Mumford stacks in general are classified and studied by Coates-Iritani-Jiang \cite{CIJ18}. In our setting, $\X_+$ and $\X_-$ differ by one of the following two transformations:
\begin{itemize}
  \item A \emph{flop}: There is a birational morphism $\phi: \X_+ \dashrightarrow \X_-$ restricting to an isomorphism between dense open subsets $U_\pm$ in $\X_\pm$, and the exceptional locus $\X_\pm \setminus U_\pm$ is a torus-invariant curve $\V(\tau_\pm^{ex})$ for some $\tau_\pm^{ex} \in \Sigma_\pm(2)$. In this case, $I_K(\X_+) = I_K(\X_-)$, and in the two corresponding triangulations of $\Delta$, the segments $|\tau_\pm^{ex}| \cap \Delta$ are the two diagonals of a common quadrilateral.

  \item A \emph{(partial) resolution}: There is a surjective morphism $\phi:\X_+ \to \X_-$ that contracts a torus-invariant divisor $\V(\rho_{i^{ex}})$ for some $i^{ex} \in \{1, \dots, R\}$ and restricts to an isomorphism from $U_+ = \X_+ \setminus \V(\rho_{i^{ex}})$ to $U_- = \phi(U_+)$. In this case, $I_K(\X_+) = I_K(\X_-) \sqcup \{i^{ex}\}$, and the triangulation of $\Delta$ corresponding to $\X_+$ refines that corresponding to $\X_-$ by a star subdivision centered at $b_{i^{ex}}$.
\end{itemize}

In addition, we need to take into account how Aganagic-Vafa branes transform along with the ambient spaces. Starting with a framed Aganagic-Vafa brane $(\L_-, f_-)$ in $\X_-$, we divide our discussion into three cases based on the location of $\L_-$ relative to the exceptional locus of the crepant transformation $\phi$. We describe these cases below and set up notations along the way. As in Section \ref{sect:AVBranes}, we denote $(\tau_-, \sigma_-)$ as the flag associated to $\L_-$; in the case $\L_-$ is inner, we denote $(\tau_-, \sigma_-')$ as the second associated flag. Let $\ell = |G_{\tau_-}|$.

\subsubsection{Case I: $\Sigma_+$ contains $\tau_-$, $\sigma_-$ (and $\sigma_-'$)}
In this case, $\L_-$ is contained in $U_-$ and induces via the isomorphism $\phi|_{U_+}$ to an Aganagic-Vafa brane $\L_+$ in $\X_+$ with the same associated flag(s). $\L_+$ has framing $f_+$ depending on $f_-$. We denote $\tau = \tau_-$, $\sigma = \sigma_-$ (and $\sigma' = \sigma_-'$). See Figure \ref{fig:WallCrossCaseI}.
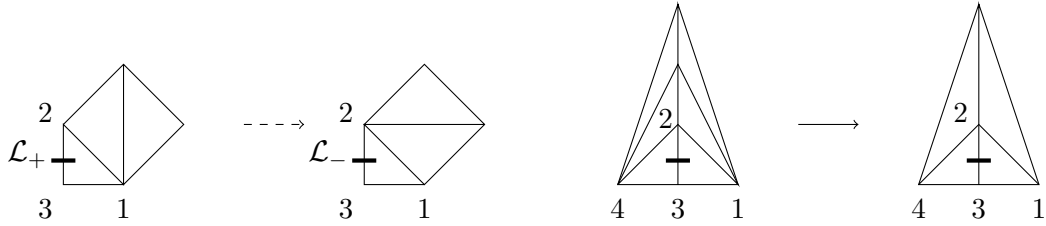
\begin{figure}[htb]
\[  \begin{tikzpicture}[scale=0.8]
    \draw (0,1) -- (1,2) -- (2,1) -- (1,0) -- (0,0) -- (0,1) -- (1,0) -- (1,2);
    \node at (1, -0.4) {$1$};
    \node at (-0.3, 1.2) {$2$};
    \node at (-0.3, -0.4) {$3$};
    \draw[dashed, ->] (3,1) -- (4,1);
    \draw (6,0) -- (5,0) -- (5,1) -- (6,2) -- (7,1) -- (6,0) -- (5,1) -- (7,1);
    \node at (6, -0.4) {$1$};
    \node at (4.7, 1.2) {$2$};
    \node at (4.7, -0.4) {$3$};
    \draw[ultra thick] (-0.2, 0.4) -- (0.2, 0.4);
    \node at (-0.6, 0.5) {$\L_+$};
    \draw[ultra thick] (4.8, 0.4) -- (5.2, 0.4);
    \node at (4.4, 0.5) {$\L_-$};
\end{tikzpicture} \qquad \qquad
\begin{tikzpicture}[scale=0.8]
    \draw (1,0) -- (1,3) -- (2,0) -- (0,0) -- (1,3);
    \draw (0,0) -- (1,1) -- (2,0) -- (1,2) -- (0,0);
    \node at (2, -0.4) {$1$};
    \node at (0.8, 1.1) {$2$};
    \node at (1, -0.4) {$3$};
    \node at (0, -0.4) {$4$};
    \draw[->] (3,1) -- (4,1);
    \draw (6,0) -- (6,3) -- (7,0) -- (5,0) -- (6,3);
    \draw (5,0) -- (6,1) -- (7,0);
    \node at (7, -0.4) {$1$};
    \node at (5.7, 1.2) {$2$};
    \node at (6, -0.4) {$3$};
    \node at (5, -0.4) {$4$};
    \draw[ultra thick] (0.8, 0.4) -- (1.2, 0.4);
    \draw[ultra thick] (5.8, 0.4) -- (6.2, 0.4);
\end{tikzpicture}
 \]
\caption{Illustration of Case I. We adopt the labeling $i_1(\tau, \sigma)=1$, $i_2(\tau, \sigma) =2$, $i_3(\tau, \sigma) = 3$ (and $i_1(\tau, \sigma') = 4$).}
\label{fig:WallCrossCaseI}
\end{figure}

\subsubsection{Case II: $\Sigma_+$ contains $\tau_-$ but not $\sigma_-$}
Geometrically, there are two subcases:
\begin{itemize}
\item Case IIa: $\phi$ is a flop, and $\tau_-^{ex}$ is a face of $\sigma_-$ other than $\tau_-$.

\item Case IIb: $\phi$ (partially) resolves the singularity at $\V(\sigma_-)$ in $\X_-$ by a star subdivision centered at some $b_{i^{ex}} \in |\sigma_-| \setminus |\tau_-|$.
\end{itemize}
In either subcase, $\L_-$ is contained in $U_-$ and induces via the isomorphism $\phi|_{U_+}$ an Aganagic-Vafa brane $\L_+$ in $\X_+$ that intersects $\V(\tau_-)$. $\L_+$ has framing $f_+$ depending on $f_-$. We denote $\tau = \tau_-$, and let $\sigma_+ \in \Sigma_+(3)$ be the cone spanned by $\tau, \tau_+^{ex}$ in Case IIa and by $\tau, \rho_{i^{ex}}$ in Case IIb.
If $\L_-$ is outer, $\L_+$ is also outer and has associated flag $(\tau, \sigma_+)$. If $\L_-$ is inner, $\L_+$ is also inner and has associated flags $(\tau, \sigma_+)$ and $(\tau, \sigma')$; in this ease we denote $\sigma'= \sigma_-'$. See Figure \ref{fig:WallCrossCaseII}.
\begin{figure}[htb]
  \[  \begin{tikzpicture}[scale=0.8]
      \draw (0,0) -- (0,1) -- (1,1) -- (1,0) -- (0,0) -- (1,1);
      \node at (1.3, -0.4) {$4$};
      \node at (1.3, 1.4) {$1$};
      \node at (-0.3, 1.4) {$2$};
      \node at (-0.3, -0.4) {$3$};
      \draw[dashed, ->] (2,0.5) -- (3,0.5);
      \draw (4.5,1) -- (5.5,1) -- (5.5,0) -- (4.5,0) -- (4.5,1) -- (5.5,0);
      \node at (5.8, -0.4) {$4$};
      \node at (5.8, 1.4) {$1$};
      \node at (4.2, 1.4) {$2$};
      \node at (4.2, -0.4) {$3$};
      \draw[ultra thick] (-0.2, 0.6) -- (0.2, 0.6);
      \node at (-0.6, 0.5) {$\L_+$};
      \draw[ultra thick] (4.3, 0.4) -- (4.7, 0.4);
      \node at (3.9, 0.5) {$\L_-$};
  \end{tikzpicture} \qquad \qquad
  \begin{tikzpicture}[scale=0.8]
      \draw (0,2) -- (-1,2) -- (0,0) -- (2,0) -- (0,2) -- (0,0) -- (1,1);
      \node at (-1.3, 2.3) {$5$};
      \node at (1.3, 1.3) {$1$};
      \node at (0.3, 2.3) {$2$};
      \node at (0, -0.4) {$3$};
      \node at (2, -0.4) {$4$};
      \draw[->] (3,1) -- (4,1);
      \draw (5.5,0) -- (5.5,2) -- (7.5,0) -- (5.5,0) -- (4.5,2) -- (5.5,2);
      \node at (4.2, 2.3) {$5$};
      \node at (5.8, 2.3) {$2$};
      \node at (5.5, -0.4) {$3$};
      \node at (7.5, -0.4) {$4$};
      \draw[ultra thick] (-0.2, 1) -- (0.2, 1);
      \draw[ultra thick] (5.3, 1) -- (5.7, 1);
  \end{tikzpicture}
   \]
\caption{Illustration of Cases IIa and IIb. We adopt the labeling $i_2(\tau, \sigma_\pm) = 2$, $i_3(\tau, \sigma_\pm) = 3$, $i_1(\tau, \sigma_+) = 1$, $i_1(\tau, \sigma_-) = 4$ (and $i_1(\tau, \sigma') = 5$).}
\label{fig:WallCrossCaseII}
\end{figure}
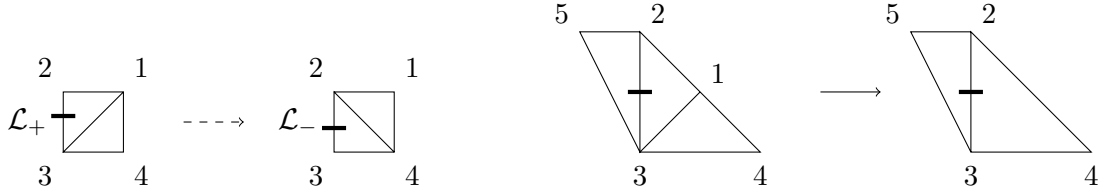

\begin{example} \rm{
The crepant resolution $K_{\P^2} \to [\C^3/\Z_3]$ (given in Example \ref{ex:KP2}), with a pair of corresponding Aganagic-Vafa branes, is another instance of Case IIb.
}\end{example}

\begin{remark} \rm{
If $\L_-$ is inner, there is a symmetric case where $\Sigma_+$ contains $\tau_-$ but not $\sigma_-'$. This is equivalent to Case II after interchanging the labels $\sigma_-$ and $\sigma_-'$.
}\end{remark}

\subsubsection{Case III: $\Sigma_+$ does not contain $\tau_-$ and $\phi$ is a (partial) resolution}
In this case, $\V(\tau_-)$ is ineffective, and $\phi$ (partially) resolves $\V(\tau_-)$ by a star subdivision centered at some $b_{i^{ex}} \in |\tau_-|$. Then $\rho_{i^{ex}}$ subdivides $\tau_-$ into two cones $\tau_+^1, \tau_+^2 \in \Sigma_+(2)$, and the preimage of $\L$ under $\phi$ consists of two Aganagic-Vafa branes $\L_+^1$ and $\L_+^2$ that intersect $\V(\tau_+^1)$ and $\V(\tau_+^2)$ respectively. $\L_+^1$ has framing $f_+$ and $\L_+^2$ has framing $f_{+,2}$, both depending on $f_-$. We set $\ell_j = |G_{\tau_+^j}|$. Then we have $\ell_1 + \ell_2 = \ell$.

The ray $\rho_{i^{ex}}$ also subdivides $\sigma_-$ (resp. $\sigma_-'$ when $\L_-$ is inner) into two cones $\sigma_+^1, \sigma_+^2 \in \Sigma_+(3)$ (resp. $\sigma_+^{1, \prime}, \sigma_+^{2,\prime} \in \Sigma_+(3)$) that contain $\tau_+^1, \tau_+^2$ respectively. For $j = 1,2$, if $\L_-$ is outer, $\L_+^j$ is also outer and has associated flag $(\tau_+^j, \sigma_+^j)$. If $\L_-$ is inner, $\L_+^j$ is also inner and has associated flags $(\tau_+^j, \sigma_+^j)$ and $(\tau_+^j, \sigma_+^{j,\prime})$. See Figure \ref{fig:WallCrossCaseIII}.

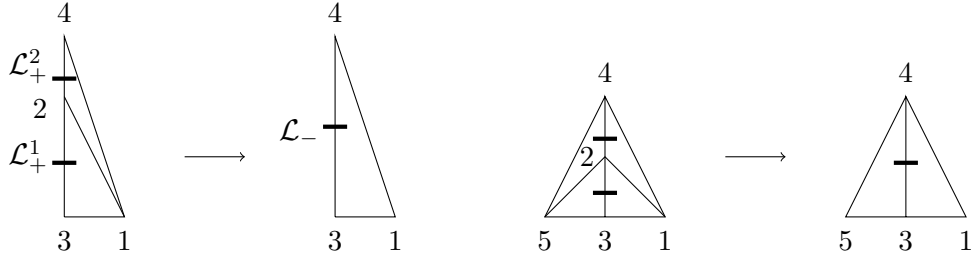
\begin{figure}[htb]
\[
\begin{tikzpicture}[scale=0.8]
    \draw (1,0) -- (1,3) -- (2,0) -- (1,0);
    \draw (1,2) -- (2,0);
    \node at (2, -0.4) {$1$};
    \node at (0.6, 1.8) {$2$};
    \node at (1, -0.4) {$3$};
    \node at (1, 3.4) {$4$};
    \draw[->] (3,1) -- (4,1);
    \draw (5.5,0) -- (5.5,3) -- (6.5,0) -- (5.5,0);
    \node at (6.5, -0.4) {$1$};
    \node at (5.5, -0.4) {$3$};
    \node at (5.5, 3.4) {$4$};
    \draw[ultra thick] (0.8, 2.3) -- (1.2, 2.3);
    \node at (0.4, 2.5) {$\L_+^2$};
    \draw[ultra thick] (0.8, 0.9) -- (1.2, 0.9);
    \node at (0.4, 1) {$\L_+^1$};
    \draw[ultra thick] (5.3, 1.5) -- (5.7, 1.5);
    \node at (4.9, 1.4) {$\L_-$};
\end{tikzpicture}\qquad \qquad
\begin{tikzpicture}[scale=0.8]
    \draw (1,0) -- (1,2) -- (2,0) -- (0,0) -- (1,2);
    \draw (0,0) -- (1,1) -- (2,0);
    \node at (2, -0.4) {$1$};
    \node at (0.7, 1) {$2$};
    \node at (1, -0.4) {$3$};
    \node at (1, 2.4) {$4$};
    \node at (0, -0.4) {$5$};
    \draw[->] (3,1) -- (4,1);
    \draw (6,0) -- (6,2) -- (7,0) -- (5,0) -- (6,2);
    \node at (7, -0.4) {$1$};
    \node at (6, -0.4) {$3$};
    \node at (6, 2.4) {$4$};
    \node at (5, -0.4) {$5$};
    \draw[ultra thick] (0.8, 1.3) -- (1.2, 1.3);
    \draw[ultra thick] (0.8, 0.4) -- (1.2, 0.4);
    \draw[ultra thick] (5.8, 0.9) -- (6.2, 0.9);
\end{tikzpicture}
\]
\caption{Illustration of Case III through the examples of $\A_1'$ and $\A_2$. We adopt the labeling $i_2(\tau_+^1, \sigma_+^1) = i_3(\tau_+^2, \sigma_+^2) = i^{ex} = 2$, $i_3(\tau_+^1, \sigma_+^1) = i_3(\tau_-, \sigma_-) = 3$, $i_2(\tau_+^2, \sigma_+^2) = i_2(\tau_-, \sigma_-) = 4$, $i_1(\tau_-, \sigma_-) = 1$ (and $i_1(\tau_-, \sigma_-')= 5$).}
\label{fig:WallCrossCaseIII}
\end{figure}

\begin{remark} \rm{
We do not consider the case where $\Sigma_+$ does not contain $\tau_-$ and $\phi$ is a flop ($\tau_- = \tau_-^{ex}$), since in this situation $\L_-$ does not naturally induce any Aganagic-Vafa brane(s) in $\X_+$.
}\end{remark}

\subsection{Open Crepant Transformation Conjecture: statement}\label{sect:OCTC}
We can now formulate the Open Crepant Transformation Conjecture for the pair of toric Calabi-Yau 3-orbifolds $\X_+$ and $\X_-$. Observe that in Cases I or II, both disk potentials $W^{\X_+, (\L_+, f_+)}$ and $W^{\X_-, (\L_-, f_-)}$ take value in $H^*_{\CR}(\B \mu_\ell; \C)$ and have $\ell$ components. As in (\ref{eq:DiskPotentialDecompose}), we write:
\[      W^{\X_+, (\L_+, f_+)}(q_+, x_+) = W_\ell^+(q_+,x_+)\one_0 + \sum_{j=1}^{\ell-1} W_j^+(q_+,x_+) \xi_\ell^{\ell-j} \one_j, \]
\[      W^{\X_-, (\L_-, f_-)}(q_-, x_-) = W_\ell^-(q_-,x_-)\one_0 + \sum_{j=1}^{\ell-1} W_j^-(q_-,x_-) \xi_\ell^{\ell-j} \one_j. \]
In Case III, $W^{\X_-, (\L_-, f_-)}$ takes value in $H^*_{\CR}(\B \mu_\ell; \C)$ and decomposes into $\ell$ components as above. Moreover, $W^{\X_+, (\L_+^1, f_+)}$ (resp. $W^{\X_+, (\L_+^2, f_{+,2})}$) takes value in $H^*_{\CR}(\B \mu_{\ell_1}; \C)$ (resp. $H^*_{\CR}(\B \mu_{\ell_2}; \C)$) and has $\ell_1$ (resp. $\ell_2$) components. There are $\ell_1 +\ell_2 = \ell$ components in total. We write:
\[      W^{\X_+, (\L_+^1, f_+)}(q_+, x_+) = W_{\ell_1}^+(q_+,x_+)\one_0 + \sum_{j=1}^{\ell_1-1} W_j^+(q_+,x_+) \xi_{\ell_1}^{\ell_1-j} \one_j, \]
\[      W^{\X_+, (\L_+^2, f_{+,2})}(q_+, x_{+,2}) = W_\ell^+(q_+,x_{+,2})\one_0 + \sum_{j=1}^{\ell_2-1} W_{\ell_1+j}^+(q_+,x_{+,2}) \xi_{\ell_2}^{\ell_2-j} \one_j. \]

Given $m \in \Z_{>0}$, let $\omega_m = \exp(\frac{2\pi\sqrt{-1}}{m})$ be the primitive $m$-th root of unity, and let $U_m$ denote the following $m$-by-$m$ matrix whose entries are $m$-th roots of unity:
\begin{equation}\label{eq:MatrixUm}
  U_m = [\omega_m^{-ij}]_{i,j}= \begin{bmatrix}
        \omega_m^{-1} & \omega_m^{-2} & \cdots & \omega_m & 1\\
        \omega_m^{-2} & \omega_m^{-4} & \cdots & \omega_m^2 & 1\\
        \vdots & \vdots & \ddots & \vdots & \vdots\\
        \omega_m & \omega_m^2 & \cdots & \omega_m^{m-1} & 1\\
        1 & 1 & \cdots & 1 & 1
  \end{bmatrix}.
\end{equation}

\begin{theorem}\label{thm:OCTC}
There is an identification of disk invariants
\begin{equation}\label{eq:OCTCThmCase12}
  \begin{bmatrix}   W_1^+\\ W_2^+ \\ \vdots \\ W_\ell^+ \end{bmatrix} = \begin{bmatrix}   W_1^-\\ W_2^- \\ \vdots \\ W_\ell^- \end{bmatrix}
\end{equation}
in Cases I or II, and an identification
\begin{equation}\label{eq:OCTCThmCase3}
  \diag(U_{\ell_1}, U_{\ell_2})^{-1}\begin{bmatrix}   \xi_{\ell_1}^{\ell_1-1}W_1^+\\ \vdots \\ \xi_{\ell_1} W_{\ell_1-1}^+ \\ W_{\ell_1}^+ \\ \xi_{\ell_2}^{\ell_2-1} W_{\ell_1+1}^+ \\ \vdots \\ W_\ell^+ \end{bmatrix} = U_\ell^{-1}\begin{bmatrix} \xi_\ell^{\ell-1} W_1^- \\ \vdots \\ \xi_\ell W_{\ell-1} \\ W_\ell^- \end{bmatrix}
\end{equation}
in Case III, where $\diag(U_{\ell_1}, U_{\ell_2})$ is the block diagonal matrix with blocks $U_{\ell_1}, U_{\ell_2}$, under relations (\ref{eq:Qchange}, \ref{eq:XYFchange}, \ref{eq:XYFchangeCase3Plus}) between open-closed moduli parameters $(q_\pm, x_\pm)$ and framings, and analytic continuation.
\end{theorem}


\begin{remark}\label{rem:KeZhou} \rm{
Ke-Zhou \cite{KZ15} studied the Open Crepant Transformation Conjecture in the particular situation where $\L_\pm$ are outer and $\X_+$ is a maximal resolution of $\X_-$ with respect to $(\X_-, \V(\tau_-))$, i.e. $\tau= \tau_-$ is still a cone in $\Sigma_+(2)$ and $\X_+$ is singular only along $\V(\tau)$. In particular, if $\L_\pm$ are effective, then $\X_+$ completely resolves the singularities of $\X_-$. The resolution $\X_+ \to \X_-$ is a composition of partial resolutions in our Cases I and IIb. \cite{KZ15} gave an explicit identification of the disk potentials $W^{\X_+, (\L_+, f_+)}$ and $W^{\X_-, (\L_-, f_-)}$ using the formula given by an earlier version of \cite{FLT19}. Their identification and framing relations given in Theorem 4.2 differ slightly from ours (\ref{eq:OCTCThmCase12}, \ref{eq:XYFchange}), which results from a constant-factor difference between the disk potential formula in \cite{FLT19} and that in the earlier version. In general, especially in Case III, the disk potentials cannot be directly identified and analytic continuation is required.
}\end{remark}

\begin{example} \label{ex:A1A2OCTC} \rm{
Consider the crepant resolution $\X_+ = K_{\P^1} \oplus \O_{\P^1} \to \A_1 = \X_-$ (see Example \ref{ex:A1SecondaryFan}). Let $(\L_-, f)$ be a framed Aganagic-Vafa brane intersecting the torus orbit corresponding to the cone spanned by $b_3$ and $b_4$, and let $\L_+^1$, $\L_+^2$ be Aganagic-Vafa branes as defined in the previous subsection. Then from Case III of Theorem \ref{thm:OCTC} we have
\[  \begin{bmatrix}
      W^{\X_+, (\L_+^1, f)} \\ W^{\X_+, (\L_+^2, f+1)}
    \end{bmatrix} =
    \begin{bmatrix}
      -1 & 1 \\ 1 & 1
    \end{bmatrix}^{-1} \begin{bmatrix}
      \sqrt{-1}W_1^{\X_-, (\L_-, f)} \\
      W_2^{\X_-, (\L_-,f)}
    \end{bmatrix}.
\]
As a second example, let $\X_- = \A_2$ and $\X_+ \to \X_-$ be the partial resolution shown in Figure \ref{fig:WallCrossCaseIII}. Pick a framing $f \in \Z$ for $\L_-$. Then from Case III of Theorem \ref{thm:OCTC} we have
\[  \begin{bmatrix}
    -1 & 1 & 0\\ 1 & 1 & 0\\ 0 & 0 & 1
  \end{bmatrix}^{-1}
  \begin{bmatrix}
      \sqrt{-1} W_1^{\X_+, (\L_+^1, f)} \\ W_2^{\X_+, (\L_+^1, f)} \\ W^{\X_+, (\L_+^2, f+2)}
    \end{bmatrix} =
    \begin{bmatrix}
      \omega_3^2 & \omega_3 & 1 \\ \omega_3 & \omega_3^2 & 1\\ 1 & 1 & 1
    \end{bmatrix}^{-1}
    \begin{bmatrix}
      \xi_3^2W_1^{\X_-, (\L_-, f)} \\
      \xi_3W_2^{\X_-, (\L_-, f)} \\
      W_3^{\X_-, (\L_-,f)}
    \end{bmatrix}.
\]
One has similar relations in the case $\X_-$ is $\A_1'$ or $\A_2'$, and $\L_-$ is an inner Aganagic-Vafa brane intersecting the torus invariant curve that has nontrivial generic stablizers.
} \end{example}

\begin{example} \label{ex:AnOCTC} \rm{
Let $\X = \A_n$ (see Example \ref{ex:An}) and $(\L,f)$ be a framed Aganagic-Vafa brane intersecting the torus orbit corresponding to the cone spanned by $(0, 0, 1)$ and $(0,n+1,1)$. Let $\Y$ be the full crepant resolution of $\X$ (given by taking star subdivisions centered at $(0,1,1), \dots, (0,n,1)$). For $j = 1, \dots, n+1$, let $\L^j$ be an Aganagic-Vafa brane in $\Y$ that intersects the torus orbit corresponding to the cone spanned by $(0, j-1, 1)$ and $(0,j, 1)$.
Then using Cases I and III of Theorem \ref{thm:OCTC}, we can show the following relation:
\[  \begin{bmatrix}
      W^{\Y, (\L^1, f)} \\ W^{\Y, (\L^2, f+1)} \\ \vdots \\ W^{\Y, (\L^{n+1}, f+n)}
    \end{bmatrix} =
    U_{n+1}^{-1} \begin{bmatrix}
      \xi_{n+1}^nW_1^{\X, (\L, f)} \\
      \vdots \\
      \xi_{n+1}W_n^{\X, (\L,f)}\\
      W_{n+1}^{\X, (\L,f)}
    \end{bmatrix}.
\]
Brini-Cavalieri-Ross \cite[Theorem 4.2]{BCR17} also related the disk invariants of $\X$ and $\Y$ relative to the Aganagic-Vafa branes above. We remark that our relation differs from theirs by a sign.
}\end{example}

\section{Mirror Curves and Mirror Symmetry for Disks}\label{sect:MCMSDisk}
In preparation for the proof of the Open Crepant Transformation Conjecture (Theorem \ref{thm:OCTC}), we define the mirror curves following \cite{FLT19, FLZ19} and study the structure of the defining equation of mirror curves. In particular, we show that a pair of toric Calabi-Yau 3-orbifolds related by toric wall-crossing have isomorphic mirror curves for generic choices of parameters (Proposition \ref{prop:MirrorcurveWallCross}). We also state the mirror theorem for disk invariants \cite{FLT19} (Theorem \ref{thm:MirrorThmDisk}) to be used later on.

\subsection{Mirror curves}\label{sect:MirrorCurve}
The mirror curve of $\X$ is an affine curve in $(\C^*)^2 = \Spec \C[\tilde{x}^{\pm1}, \tilde{y}^{\pm1}]$ parametrized by closed moduli parameters $q = (q_1, \dots, q_k) \in \C^k$. To give the defining equation, we first choose lattice vectors $p_1, \dots, p_k \in \Lat^\vee \cap \enef(\X)$ such that
\begin{enumerate}[label=(\roman*)]
\item $\{p_1, \dots, p_k\}$ forms a $\Q$-basis for $\Lat^\vee_\Q:= \Lat^\vee \otimes \Q$.
\item There is a subset $A_K = A_K(\X) \subseteq \{1, \dots, k\}$ of indices such that under the identification (\ref{eq:LIdentification}), the image of $\{p_a \mid a \in A_K\}$ forms a $\Q$-basis for $H^2(\X; \Q)$.
\item There is a bijection $\iota: I_\orb \to A_\orb = A_\orb(\X) := \{1, \dots, k\} \setminus A_K$ such that $D_i = p_{\iota(i)}$ for all $i \in I_\orb$. \label{cond:PaOrbChoice}
\item For any $a \in A_K$ and $i \in I_\orb$, $p_a - D_i \not \in \enef(\X)$. \label{cond:PaPrimitive}
\end{enumerate}
We may think of $\{p_a \mid a \in A_K\}$ as K\"ahler parameters of $\X$ and $\{p_a \mid a \in A_\orb\}$ as extended K\"ahler parameters of $\X$ corresponding to the twisted sectors. We remark that condition \ref{cond:PaPrimitive} is not required by \cite{FLT19, FLZ19} but imposed here to simplify the mirror curve equation (see Lemma \ref{lem:MCOrbCoeff}).

For a cone $\sigma \in \Sigma(3)$, since $\Lat^\vee \cap \enef(\X) \subseteq \enef_\sigma = \sum_{i \in I_\sigma} \R_{\ge 0} D_i$, we may write
\begin{equation}\label{eq:PaLinearExpansion}
  p_a = \sum_{i \in I_\sigma}^R s_{ai}^\sigma D_i,
\end{equation}
where each $s_{ai}^\sigma$ is a nonnegative rational number. For $i \in I'_\sigma$, we set $s_{ai}^\sigma = 0$ for all $a$. Now for each $i \in \{1, \dots, R\}$, let
\[  s_i^\sigma(q) := \prod_{a = 1}^k q_a^{s_{ai}^\sigma}. \]
In particular, $s_i^\sigma(q) = 1$ for each $i \in I_\sigma'$. By rescaling the $p_a$'s if necessary, we may assume that $s_{ai}^\sigma$ is a nonnegative integer for each $\sigma, a, i$. This makes $s_i^\sigma(q)$ a monomial in $q$ for each $\sigma, i$.

Let $(\tau, \sigma) \in F(\Sigma)$. For $q = (q_1, \dots, q_k) \in \C^k$, the \emph{mirror curve} $C_q$ of $\X$ is defined by
\begin{equation}\label{eq:MirrorcurveDefFlag}
    H(q,\tilde{x}_{(\tau,\sigma)}, \tilde{y}_{(\tau,\sigma)}) := \sum_{i = 1}^R s_i^\sigma(q)\tilde{x}_{(\tau,\sigma)}^{m_i(\tau,\sigma)}\tilde{y}_{(\tau,\sigma)}^{n_i(\tau, \sigma)} = 0.
\end{equation}

\begin{example}\label{ex:A1MC} \rm{
For $\A_1$ (see Example \ref{ex:A1SecondaryFan}), with $\tau$ as the cone spanned by $b_3$ and $b_4$, $\sigma$ as the cone spanned by $\tau$ and $b_1$, and $p_1 = D_2 = -2$, the mirror curve equation is
\[  H(q, \tilde{x}, \tilde{y}) = \tilde{x} + q_1\tilde{y} + 1 + \tilde{y}^2.  \]
}\end{example}

As mentioned in \cite[Section 4.1]{FLZ19}, choosing a different flag $(\tau', \sigma') \in F(\Sigma)$ gives a reparametrization of $C_q$. The change of basis matrix from $\{v_1(\tau, \sigma), v_2(\tau, \sigma), v_3(\tau, \sigma)\}$ to $\{v_1(\tau', \sigma'), v_2(\tau', \sigma')$, $v_3(\tau', \sigma')\}$ has form
\[  \begin{bmatrix}
      a & c & m_{i_3(\tau, \sigma)}(\tau', \sigma')\\
      b & d & n_{i_3(\tau, \sigma)}(\tau', \sigma')\\
      0 & 0 & 1
    \end{bmatrix} \in \rm{SL}(3;\Z).
\]
Under the reparametrization
\begin{equation}\label{eq:XYFlagChange}
\begin{aligned}
  \tilde{x}_{(\tau, \sigma)} &= \tilde{x}_{(\tau', \sigma')}^a \tilde{y}_{(\tau', \sigma')}^bs_{i_1(\tau,\sigma)}^{\sigma'}(q)^{w_1(\tau, \sigma)}s_{i_2(\tau,\sigma)}^{\sigma'}(q)^{w_2(\tau, \sigma)}s_{i_3(\tau,\sigma)}^{\sigma'}(q)^{w_3(\tau, \sigma)},\\
  \tilde{y}_{(\tau, \sigma)} &= \tilde{x}_{(\tau', \sigma')}^c \tilde{y}_{(\tau', \sigma')}^d \big(s_{i_2(\tau,\sigma)}^{\sigma'}(q) / s_{i_3(\tau,\sigma)}^{\sigma'}(q) \big)^{\frac{1}{\ell(\tau, \sigma)}},
\end{aligned}
\end{equation}
where
\[  w_1(\tau, \sigma) = \frac{1}{r(\tau, \sigma)}, \qquad w_2(\tau, \sigma) = \frac{s(\tau, \sigma)}{r(\tau, \sigma)\ell(\tau, \sigma)}, \qquad w_3(\tau, \sigma) = -w_1(\tau, \sigma)-w_2(\tau, \sigma), \]
we have a term-by-term identification of mirror curve equations
\begin{equation}\label{eq:MirrorcurveFlagChange}
H(q, \tilde{x}_{(\tau', \sigma')}, \tilde{y}_{(\tau', \sigma')}) = s_{i_3(\tau,\sigma)}^{\sigma'}(q)\tilde{x}_{(\tau', \sigma')}^{m_{i_3(\tau,\sigma)}(\tau', \sigma')}\tilde{y}_{(\tau', \sigma')}^{n_{i_3(\tau,\sigma)}(\tau', \sigma')}H(q, \tilde{x}_{(\tau, \sigma)}, \tilde{y}_{(\tau, \sigma)}).
\end{equation}

\begin{example} \label{ex:A2MC} \rm{
Let $\X$ be the partial resolution of $\A_2$ given in Figure \ref{fig:WallCrossCaseIII}. The mirror curve equations of $\X$ with respect to the flags $(\tau_+^1, \sigma_+^1), (\tau_+^2, \sigma_+^2)$ are
\[  H_1(q, \tilde{x}_1, \tilde{y}_1) = \tilde{x}_1 + \tilde{y}_1^2 + 1 + q_1 \tilde{y}_1^3 + q_2 \tilde{y}_1, \quad  H_2(q, \tilde{x}_2, \tilde{y}_2) = \tilde{x}_2 + 1 + q_1^2\tilde{y}_2^{-2} + \tilde{y}_2 + q_1q_2\tilde{y}_2^{-1}.  \]
We have an identification $H_1 = \tilde{y}_1^2H_2$ under the reparametrization
\[  \tilde{x}_2 = \tilde{x}_1\tilde{y}_1^{-2}, \qquad \tilde{y}_2 = \tilde{y}_1q_1.  \]
}\end{example}

By condition \ref{cond:PaOrbChoice}, for each $i \in I_\orb$ and $\sigma \in \Sigma(3)$, we have $s_{\iota(i)i}^\sigma = 1$ and $s_{\iota(i)i'}^\sigma = 0$ for any $i' \neq i$. In other words, $q_{\iota(i)}$ appears only in the monomial $s_{i}^\sigma$ and with exponent 1. Combining this with condition \ref{cond:PaPrimitive} and the identification (\ref{eq:MirrorcurveFlagChange}), we have the following simple characterization of the monomial $s_i^\sigma$'s for $i \in I_\orb$:

\begin{lemma}\label{lem:MCOrbCoeff}
Let $i \in I_\orb$ and $\sigma \in \Sigma(3)$ such that $b_i \in |\sigma|$. Then
\[  s_i^\sigma(q) = q_{\iota(i)}.  \]
In general, assume that $I'_{\sigma} = \{i_1, i_2, i_3\}$, and write
\[  b_i = c_1b_{i_1} + c_2b_{i_2} + c_3b_{i_3}  \]
for $c_1, c_2, c_3 \in [0,1]$ with $c_1 + c_2+c_3 = 1$. Then for any $\sigma' \in \Sigma(3)$, we have
\begin{equation}\label{eq:MCOrbCoeffGenral}
  s_i^{\sigma'}(q) = q_{\iota(i)}s_{i_1}^{\sigma'}(q)^{c_1}s_{i_2}^{\sigma'}(q)^{c_2}s_{i_3}^{\sigma'}(q)^{c_3}.
\end{equation}
\end{lemma}

\begin{proof}
Since for each $i' \in I_\orb$ we have $s^\sigma_{\iota(i')i} = 1$ if and only if $i' = i$, we can write
\[  s_i^\sigma(q) = q_{\iota(i)}s(q)  \]
for some monomial $s(q)$ in the variables $\{q_a \mid a \in A_K\}$ only. We need to show that $s(q) = 1$. Suppose otherwise that $s(q)$ contains a factor $q_a$ for some $a \in A_K$. This means that in the expansion (\ref{eq:PaLinearExpansion}) of $p_a$, the coefficient $s_{ai}^\sigma$ of $D_i$ is positive. Thus we see that $p_a - D_i \in \enef_\sigma$. We will show that in fact $p_a - D_i \in \enef_{\sigma'}$ for all $\sigma' \in \Sigma(3)$, thus arriving at a contradiction with condition \ref{cond:PaPrimitive}.

We pick a face $\tau \in \Sigma(2)$ of $\sigma$ and without loss of generality assume that $i_1 = i_1(\tau, \sigma), i_2 = i_2(\tau, \sigma), i_3 = i_3(\tau, \sigma)$. Note that
\[  m_i(\tau,\sigma) = c_1r(\tau, \sigma), \qquad n_i(\tau, \sigma) = -c_1s(\tau, \sigma)+ c_2\ell(\tau, \sigma).  \]
Given any $\sigma' \in \Sigma(3)$ and face $\tau' \in \Sigma(2)$ of $\sigma'$, as part of the identification (\ref{eq:MirrorcurveFlagChange}), we have the relation
\begin{equation}\label{eq:MCOrbCoeffFlagChange}
\begin{aligned}
  s_i^{\sigma'}(q) &= s_i^\sigma(q) s_{i_1}^{\sigma'}(q)^{m_i(\tau,\sigma)w_1(\tau,\sigma)}s_{i_2}^{\sigma'}(q)^{m_i(\tau,\sigma)w_2(\tau,\sigma)+\frac{n_i(\tau,\sigma)}{\ell(\tau,\sigma)}}s_{i_3}^{\sigma'}(q)^{1+m_i(\tau,\sigma)w_3(\tau,\sigma)-\frac{n_i(\tau,\sigma)}{\ell(\tau,\sigma)}}\\
  &= q_{\iota(i)}s(q)s_{i_1}^{\sigma'}(q)^{c_1}s_{i_2}^{\sigma'}(q)^{c_2}s_{i_3}^{\sigma'}(q)^{c_3}.
\end{aligned}
\end{equation}
As a result, $q_a$ also appears as a factor of $s_i^{\sigma'}(q)$. Similar to above, this implies that in the expansion (\ref{eq:PaLinearExpansion}) of $p_a$ with respect to the basis $\{D_{i'} \mid i' \in I_{\sigma'}\}$, the coefficient $s_{ai}^{\sigma'}$ of $D_i$ is positive. Hence $p_a - D_i \in \enef_{\sigma'}$.

Therefore, we conclude that $s(q)=1$. The more general statement (\ref{eq:MCOrbCoeffGenral}) then follows from relation (\ref{eq:MCOrbCoeffFlagChange}).
\end{proof}

\subsection{Incorporating the open moduli parameter}\label{sect:MirrorcurveOpen}
Based on the framed Aganagic-Vafa brane $(\L,f)$, we reparametrize the mirror curve $C_q$ of $\X$ by incorporating the open moduli parameter $x$. As in Section \ref{sect:AVBranes}, if $\L$ is outer, we can associate to it a flag $(\tau, \sigma) \in F(\Sigma)$ such that $\L$ intersects $\V(\tau)$ and $\sigma$ contains $\tau$; similarly, if $\L$ is inner, we can associate to it two flags $(\tau, \sigma), (\tau, \sigma') \in F(\Sigma)$. In either case, we use the basis $\{v_1(\tau, \sigma), v_2(\tau, \sigma), v_3(\tau, \sigma)\}$ of $N$ associated to the flag $(\tau,\sigma)$ and its dual basis $\{u_1(\tau, \sigma), u_2(\tau, \sigma), u_3(\tau, \sigma)\}$ of $M$ as preferred bases of $N$ and $M$ for $(\X, (\L,f))$. For simplicity, throughout this subsection and the next, we denote
\[  v_j = v_j(\tau, \sigma), \quad u_j = u_j(\tau, \sigma), \quad i_j = i_j(\tau, \sigma) \qquad \mbox{for } j = 1,2,3;    \]
\[  r = r(\tau, \sigma), \quad s = s(\tau, \sigma), \quad \ell = \ell(\tau, \sigma);  \]
\[  m_i = m_i(\tau, \sigma), \quad n_i = n_i(\tau, \sigma) \qquad \mbox{for } i = 1, \dots, R. \]
In addition, if $\L$ is inner, we denote $i_4 = i_1(\tau, \sigma')$ and $r' = r(\tau, \sigma')$. Note that $r' = -m_{i_4} > 0$.


With the parameters
\[  s_{ai} = s_{ai}^\sigma, \qquad s_i(q) = s_i^\sigma(q), \qquad  \tilde{x} = \tilde{x}_{(\tau, \sigma)}, \qquad \tilde{y} = \tilde{y}_{(\tau, \sigma)} \]
associated to the preferred flag $(\tau, \sigma)$, the mirror curve equation (\ref{eq:MirrorcurveDefFlag}) becomes
\begin{equation}\label{eq:MirrorcurveDef}
  0 = H(q, \tilde{x}, \tilde{y}) = \tilde{x}^r\tilde{y}^{-s} + \tilde{y}^\ell + 1 + \sum_{i \in I_\sigma} s_i(q)\tilde{x}^{m_i}\tilde{y}^{n_i}.
\end{equation}
Moreover, using the framing $f$, we perform the reparametrization
\begin{equation}\label{eq:MirrorcurveXYReparam}
     \tilde{x} = xy^{-f}, \qquad  \tilde{y}=y
\end{equation}
and turn (\ref{eq:MirrorcurveDef}) into
\begin{equation}\label{eq:MirrorcurveDefOpen}
  0 = H(q, x, y) = x^ry^{-rf-s} + y^\ell + 1 + \sum_{i \in I_\sigma} s_i(q)x^{m_i}y^{-m_if +n_i}.
\end{equation}

Equation (\ref{eq:MirrorcurveDefOpen}) can be rewritten as
\begin{equation}\label{eq:MirrorcurveIncludeX}
  0 = H(q, x, y) = \sum_{i = 1}^R \tilde{s}_i(q,x)y^{-m_if +n_i} = x^ry^{-rf-s} + y^\ell + 1 + \sum_{i \in I_\sigma} \tilde{s}_i(q,x)y^{-m_if +n_i} ,
\end{equation}
where for each $i$,
\[  \tilde{s}_i(q,x) = s_i(q)x^{m_i}.  \]
If $\L$ is outer, we have $m_i \ge 0$ for all $i = 1, \dots, R$. Thus each $\tilde{s}_i(q,x)$ is a monomial in $q, x$. If $\L$ is inner, some of the $m_i$'s are negative, yet we have the following observation on the corresponding $s_i(q)$'s:

\begin{lemma}\label{lem:MCNegCoeff}
Let $\L$ be an inner brane. Recall that $I'_{\sigma'} = \{i_2,i_3,i_4\}$ and $r' = -m_{i_4}$. Then $q^{\alpha} := s_{i_4}(q)^{1/r'}$ is a monomial in $\{q_a \mid a \in A_K\}$. Moreover, for each $i =1, \dots, R$ such that $m_i<0$, $(q^\alpha)^{-m_i}$ divides $s_i(q)$.
\end{lemma}

\begin{proof}
Note that $s_{i_4}(q)$ is a monomial in $\{q_a \mid a \in A_K\}$. We first check that $q^\alpha$ is a monomial. It suffices to consider the case $r' \ge 2$. Take $i \in I_\orb$ such that $b_i \in |\sigma'|$ and $m_i = -1$. Then we can write
\[  b_i = c_2 b_{i_2} + c_3 b_{i_3} + \frac{1}{r'}b_{i_4} \]
for some $c_2, c_3 \in [0,1]$ adding up to $1- \frac{1}{r'}$. Lemma \ref{lem:MCOrbCoeff} implies that $s_i(q) = q_{\iota(i)}s_{i_4}(q)^{1/r'} = q_{\iota(i)}q^\alpha$. Since $s_i(q)$ is a monomial, $q^\alpha$ is also a monomial.

To prove the rest of the lemma, we observe that under the flag change from $(\tau, \sigma)$ to $(\tau, \sigma')$, the reparametrization (\ref{eq:XYFlagChange}, \ref{eq:MirrorcurveFlagChange}) in particular implies that
\[  s_i^{\sigma'}(q) = s_i(q)s_{i_1}^{\sigma'}(q)^{\frac{m_i}{r}} \]
for each $i = 1, \dots, R$. Taking $i = i_4$, we have $s_{i_1}^{\sigma'}(q) = s_{i_4}(q)^{\frac{r}{r'}} = (q^\alpha)^r$. Therefore for each $i$ such that $m_i<0$, we have the factorization
\[  s_i(q) = s_i^{\sigma'}(q)s_{i_1}^{\sigma'}(q)^{-\frac{m_i}{r}} = (q^\alpha)^{-m_i}s_i^{\sigma'}(q). \qedhere \]
\end{proof}

Hence for each $i$ such that $m_i<0$, $(q^\alpha x^{-1})^{-m_i}$ divides $s_i(q)x^{m_i}$. If we pick $a_0 \in A_K$ such that $q_{a_0}$ divides $q^\alpha$, then each $\tilde{s}_i(q,x)$ is a monomial in $q_1, \dots, q_{a_0-1}, q_{a_0+1}, \dots, q_k, x, q_{a_0}x^{-1}$.

\begin{remark} \label{rem:QAlphaFLT} \rm{
The notation $q^\alpha$ is taken from \cite{FLT19}. Let $\alpha = [\V(\tau)] \in H_2(\X;\Z)$ be the class of the $\T$-invariant curve $\V(\tau)$, which can also be viewed as an effective curve class in the extended Mori cone of $\X$ in $\Lat_\Q= \Lat \otimes \Q$ (see \cite[Section 2.5]{FLT19}). Under the definition of \cite{FLT19},
\[  q^\alpha:= \prod_{a = 1}^k q_a^{\inner{p_a, \alpha}},  \]
where $\inner{-,-}:\Lat^\vee_\Q \times \Lat_\Q \to \Q$ is the natural pairing. Computing the intersection product (see \cite[Section 5.1]{Fulton93}) between the Poincar\'e dual of $\alpha$ and the classes of the divisors $V(\rho_i)$'s, we obtain
\[  \inner{D_i, \alpha} = \begin{cases}
          \frac{1}{r} &\mbox{if } i = i_1,\\
          \frac{1}{r'} &\mbox{if } i = i_4,\\
          0 & \mbox{otherwise}.
        \end{cases} \]
Then for each $a = 1, \dots, k$,
\[  \inner{p_a, \alpha} = \inner{s_{ai_4}D_{i_4}, \alpha} = \frac{s_{ai_4}}{r'}.  \]
This implies that $q^\alpha = s_{i_4}(q)^{1/r'}$, and that our definition for $q^\alpha$ is consistent with \cite{FLT19}. Since $\alpha$ is an integral class and each $p_a \in \Lat^\vee$ is a lattice vector, we have that $\inner{p_a, \alpha} \in \Z_{\ge 0}$ for each $a$ and thus $q^\alpha$ is a monomial in $q$. In fact, the rest of Lemma \ref{lem:MCNegCoeff} can be obtained from the construction of the disk potential $W^{\X, (\L, f)}$ in \cite{FLT19} and the mirror theorem for disks (Theorem \ref{thm:MirrorThmDisk}).
}\end{remark}

\subsection{The mirror theorem for disks}\label{sect:MirrorThmDisk}
In this subsection, we state the mirror theorem of \cite{FLT19} that relates the disk potential of $(\X, (\L,f))$ to the mirror curve $C_q$.

Locally on $C_q$, we may solve for $y$ in the equation (\ref{eq:MirrorcurveIncludeX}) in terms of $q$ and $x$. If $\L$ is outer, in the limit $q \to 0, x \to 0$, there are $\ell$ local solutions $\kappa_1, \dots, \kappa_\ell$ to $H(y)=0$ that satisfy
\[  \log \kappa_j = \frac{\pi\sqrt{-1}}{\ell}(-1+2j) + \frac{V_j(q, x)}{\ell},  \]
where each $V_j$ is a series in $x\C[[q,x]]$. Similarly, if $\L$ is inner, in the limit $q \to 0, x \to 0, q^\alpha x^{-1} \to 0$, there are $\ell$ local solutions $\kappa_1, \dots, \kappa_\ell$ to $H(y)=0$ that satisfy
\[  \log \kappa_j = \frac{\pi\sqrt{-1}}{\ell}(-1+2j) + \frac{V_j(q, x)}{\ell},  \]
where each $V_j$ is the sum of a series in $x\C[[q,x]]$ and a series in $x^{-1}\C[[q,x^{-1}]]$.

\begin{example} \rm{
As in Example \ref{ex:A2MC}, let $\X_+$ be the partial resolution of $\A_2$ given in Figure \ref{fig:WallCrossCaseIII}. With respect to the framed Aganagic-Vafa brane $(\L_+^1, f_+)$, the mirror curve equation is
\[  H(q, x, y) = xy^{-f_+} + y^2 + 1 + q_1y^3 + q_2y = 0.  \]
There are two roots $\kappa_1, \kappa_2$ to $H(y) = 0$ that converge in the limit $q \to 0, x \to 0$. They satisfy
\[  \kappa_1 \to \sqrt{-1}, \quad \log \kappa_1 \to \frac{\pi\sqrt{-1}}{2}; \qquad \kappa_2 \to -\sqrt{-1}, \quad \log \kappa_2 \to \frac{3\pi\sqrt{-1}}{2}.  \]
Moreover, there is a third root $\kappa_3$ that satisfies the asymptotics $\kappa_3 \sim -\frac{1}{q_1}$.
}\end{example}

\begin{remark} \label{rem:LRL} \rm{
The limit $q_a \to 0$ for all $a \in A_K$ is referred to as the \emph{large radius limit} in the literature. For instance, if $\L$ is inner and we take the definition of $q^\alpha$ in \cite{FLT19} (see Remark \ref{rem:QAlphaFLT}), then roughly speaking, $\log |q^\alpha|$ is the negative of the symplectic area of the curve $V(\tau) \subset X$. Taking the limit $q_a \to 0$ for all $a \in A_K$ corresponds to letting the area of $V(\tau)$, and similarly that of any curve in $X$ representing a nonzero effective curve class, tend to infinity. In the framework of \cite{FLT19}, the limit on the open moduli parameter has a similar geometric meaning. Recall from Figure \ref{fig:AVBranes} and Section \ref{sect:DiskInvariants} that $L \cap V(\tau)$ bounds a disk $D$ in $V(\tau)$ if $\L$ is outer, and two disks $D$ and $D'$ if $\L$ is inner. Then taking the limit $x \to 0$ corresponds to letting the area of $D$ tend to infinity. In the case $\L$ is inner, taking the limit $q^\alpha x^{-1} \to 0$ corresponds to letting the area of $D'$ tend to infinity.
}\end{remark}

The mirror theorem for disks can be stated as follows (see \cite[Theorem 4.5]{FLT19}):
\begin{theorem}[Fang-Liu-Tseng \cite{FLT19}]\label{thm:MirrorThmDisk}
The following relation holds:
\begin{equation}\label{eq: MirrorThmDisk}
  \left( x\frac{\partial}{\partial x} \right)^2
  \begin{bmatrix}   \xi_\ell^{\ell-1}W_1\\ \vdots \\ \xi_\ell W_{\ell-1} \\ W_\ell \end{bmatrix}
    = U_\ell \left( x\frac{\partial}{\partial x} \right)
    \begin{bmatrix} \log \kappa_1\\ \vdots \\ \log \kappa_{\ell-1} \\ \log \kappa_\ell \end{bmatrix},
\end{equation}
where $U_\ell$ is the matrix defined in (\ref{eq:MatrixUm}).
\end{theorem}

\subsection{Mirror curves and wall-crossing}\label{sect:MirrorcurveWallCross}
To conclude our discussion on mirror curve equations and set out to prove Theorem \ref{thm:OCTC}, we give an identification of mirror curve equations of a pair of toric Calabi-Yau 3-orbifolds $\X_+$ and $\X_-$ that differ by a single wall-crossing in the secondary fan. We treat the three cases described in Section \ref{sect:GLSM} separately, and use the notation introduced there.

Recall that $W = \enef(\X_+) \cap \enef(\X_-)$ denotes the wall between the extended Nef cones of $\X_\pm$. Following \cite[Section 5.3]{CIJ18}, we choose vectors $p_2, \dots, p_k \in \Lat^\vee \cap W$ that are linearly independent over $\Q$ and vectors $p_1^\pm \in \Lat^\vee \cap (\enef(\X_\pm) \setminus W)$, so that the basis $p^\pm = \{p_1^\pm, p_2, \dots, p_k\}$ satisfies the conditions in Section \ref{sect:MirrorCurve} and can be used to define the mirror curve of $\X_\pm$. We can write $p_1^+$ in the basis $p^-$ as
\begin{equation}\label{eq:POnePlus}
p_1^+ = -c_1p_1^- + c_2p_2 + \cdots + c_kp_k
\end{equation}
for some $c_1 \in \Q_{>0}$, $c_2, \dots, c_k \in \Q$. Based on this expression, we introduce the following relation between closed moduli parameters $q_\pm = (q_{\pm, 1}, \dots, q_{\pm, k})$:
\begin{equation}\label{eq:Qchange}
q_{-,1} = q_{+,1}^{-c_1}, \qquad q_{-,a} = q_{+,a}q_{+,1}^{c_a} \quad \mbox{for $a = 2, \dots, k$}.
\end{equation}


Let $H_+(q_+, x_+, y_+)$ and $H_-(q_-, x_-, y_-)$ denote:
\begin{itemize}
  \item In Case I, the mirror curve equations of $\X_\pm$ (given in (\ref{eq:MirrorcurveDefOpen})) defined with respect to bases $p^\pm$, the flag $(\tau, \sigma)$, and framed Aganagic-Vafa branes $(\L_\pm, f_\pm)$.
  \item In Case II, the mirror curve equations of $\X_\pm$ defined with respect to $p^\pm$, $(\tau, \sigma_\pm)$, and $(\L_\pm, f_\pm)$.
  \item In Case III, the mirror curve equation of $\X_+$ defined with respect to $p^+$, $(\tau_+^1, \sigma_+^1)$, and $(\L_+^1, f_+)$, and the mirror curve equation of $\X_-$ defined with respect to $p^-$, $(\tau_-, \sigma_-)$, and $(\L_-,f_-)$.
\end{itemize}
As in (\ref{eq:MirrorcurveDefOpen}, \ref{eq:MirrorcurveIncludeX}), for $i = 1, \dots, R$, the $i$-th term of $H_\pm(q_\pm, x_\pm, y_\pm)$ has form
\begin{equation}\label{eq:CoeffWallCrossForm}
    s_i^\pm(q_\pm)x_\pm^{m_i^\pm}y_{\pm}^{-m_i^\pm f_{\pm} + n_i^\pm} = \tilde{s}_i^\pm(q_\pm)y_{\pm}^{-m_i^\pm f_{\pm} + n_i^\pm}.
\end{equation}

\begin{proposition}\label{prop:MirrorcurveWallCross}
There is a term-by-term identification
\begin{equation}\label{eq:MirrorcurveWallCross}
    H_+(q_+, x_+, y_+) = H_-(q_-, x_-, y_-)
\end{equation}
under relations (\ref{eq:Qchange}, \ref{eq:XYFchange}) between parameters $(q_\pm, x_\pm, y_\pm, f_\pm)$.
\end{proposition}

\subsubsection{Proof of Proposition \ref{prop:MirrorcurveWallCross}, Case I}
In this case, we have $m_i^\pm = m_i(\tau, \sigma), n_i^\pm = n_i(\tau, \sigma)$ for each $i = 1, \dots, R$. Since both $\enef(\X_\pm)$ are contained in $\enef_{\sigma} = \sum_{i = 4}^R \R_{\ge 0} D_i$, as in $(\ref{eq:PaLinearExpansion})$, we write
\[  p_1^+ = \sum_{i = 4}^R s_{1i}^+ D_i, \quad p_1^- = \sum_{i = 4}^R s_{1i}^- D_i, \quad p_a = \sum_{i = 4}^R s_{ai}D_i \quad \mbox{for $a \ge 2$},  \]
and take $s_{1i}^\pm = s_{ai} = 0$ for all $i = 1,2,3, a \ge 2$. From this we have
\[  -c_1p_1^- + c_2p_2 + \cdots + c_kp_k = \sum_{i=4}^R \left(-c_1 s_{1i}^- + \sum_{a = 2}^k c_a s_{ai} \right) D_i. \]
Then (\ref{eq:POnePlus}) implies the relation
\[  \begin{bmatrix}
s_{14}^- & s_{24} & \cdots & s_{k4}\\
s_{15}^- & s_{25} & \cdots & s_{k5}\\
\vdots & \vdots & \ddots & \vdots\\
s_{1R}^- & s_{2R} & \cdots & s_{kR}
\end{bmatrix} \begin{bmatrix}
-c_1 \\ c_2 \\ \vdots \\ c_k
\end{bmatrix} = \begin{bmatrix}
  s_{14}^+ \\ s_{15}^+ \\ \vdots \\ s_{1R}^+
\end{bmatrix}.  \]
Adding in the relation (\ref{eq:Qchange}), we see that for each $i = 1, \dots, R$,
\[  s_i^-(q_-) = q_{-,1}^{s_{1i}^-}\prod_{a = 2}^k q_{-,a}^{s_{ai}} = q_{+,1}^{-c_1s_{i1}^- + \sum_{a = 2}^k c_as_{ai}} \prod_{a=2}^k q_{+,a}^{s_{ai}} = q_{+,1}^{s_{1i}^+} \prod_{a=2}^k q_{+,a}^{s_{ai}} = s_i^+(q_+). \]
The desired identification (\ref{eq:MirrorcurveWallCross}) thus follows from the relations
\begin{equation}\label{eq:XYFchangeCase1}
  x_- = x_+, \qquad y_- = y_+, \qquad f_- = f_+.
\end{equation}

\subsubsection{Proof of Proposition \ref{prop:MirrorcurveWallCross}, Case II}
In this case, the change of basis matrix from $\{v_1(\tau, \sigma_-)$, $v_2(\tau, \sigma_-), v_3(\tau, \sigma_-)\}$ to $\{v_1(\tau, \sigma_+), v_2(\tau, \sigma_+), v_3(\tau, \sigma_+)\}$ has form
\[  \begin{bmatrix}
      1 & 0 & 0\\
      b & 1 & 0\\
      0 & 0 & 1
    \end{bmatrix}
\]
for some $b \in \Z$. For $i = 1, \dots, R$, we have
\begin{equation}\label{eq:MNchangeCase2}
m_i^+ = m_i(\tau, \sigma_+) = m_i(\tau, \sigma_-) = m_i^-, \qquad n_i^+ = n_i(\tau, \sigma_+) = bm_i(\tau, \sigma_-) + n_i(\tau, \sigma_-) = bm_i^- + n_i^-.
\end{equation}
Set $m_i = m_i^\pm$. The short exact sequence (\ref{eq:StackyFanSq}) implies
\[  m_1D_1 + \sum_{i=4}^R m_iD_i = 0,  \]
that is,
\begin{equation}\label{eq:DchangeCase2}
  D_4 = - \frac{m_1}{m_4} D_1 - \sum_{i = 5}^R \frac{m_i}{m_4}D_i.
\end{equation}
Since $\enef(\X_+) \subseteq \enef_{\sigma_+} = \sum_{i = 4}^R \R_{\ge 0} D_i$ and $\enef(\X_-) \subseteq \enef_{\sigma_-} = \R_{\ge 0}D_1 + \sum_{i = 5}^R \R_{\ge 0} D_i$, we write
\[  p_1^+ = \sum_{i = 4}^R s_{1i}^+ D_i, \quad p_1^- = s_{11}^- + \sum_{i = 5}^R s_{1i}^- D_i, \quad p_a = s_{a1}D_1 + \sum_{i = 5}^R s_{ai}D_i \quad \mbox{for $a \ge 2$},  \]
and take $s_{11}^+ = s_{12}^\pm = s_{13}^\pm = s_{14}^- = 0$, $s_{ai} = 0$ for all $i = 2,3,4, a \ge 2$. From this and (\ref{eq:DchangeCase2}) we have
\[  p_1^+ = -\frac{s_{14}^+m_1}{m_4} D_1 + \sum_{i = 5}^R \left(s_{1i}^+ -\frac{s_{14}^+m_i}{m_4}\right)D_i, \]
\[  -c_1p_1^- + c_2p_2 + \cdots + c_kp_k = \left(-c_1 s_{11}^- + \sum_{a = 2}^k c_a s_{a1} \right) D_1 + \sum_{i=5}^R \left(-c_1 s_{1i}^- + \sum_{a = 2}^k c_a s_{ai} \right) D_i. \]
Then (\ref{eq:POnePlus}) implies the relation
\[  \begin{bmatrix}
s_{11}^- & s_{21} & \cdots & s_{k1}\\
s_{15}^- & s_{25} & \cdots & s_{k5}\\
\vdots & \vdots & \ddots & \vdots\\
s_{1R}^- & s_{2R} & \cdots & s_{kR}
\end{bmatrix} \begin{bmatrix}
-c_1 \\ c_2 \\ \vdots \\ c_k
\end{bmatrix} = \begin{bmatrix}
  0 \\ s_{15}^+ \\ \vdots \\ s_{1R}^+
\end{bmatrix} - \frac{s_{14}^+}{m_4}\begin{bmatrix}
  m_1 \\ m_5 \\ \vdots \\ m_R
\end{bmatrix}.  \]
Adding in the relation (\ref{eq:Qchange}), we see that for each $i = 1, \dots, R$,
\[  s_i^-(q_-) = q_{-,1}^{s_{1i}^-}\prod_{a = 2}^k q_{-,a}^{s_{ai}} = q_{+,1}^{-c_1s_{1i}^- + \sum_{a = 2}^k c_as_{ai}} \prod_{a=2}^k q_{+,a}^{s_{ai}} = q_{+,1}^{s_{1i}^+ - \frac{s_{14}^+m_i}{m_4}} \prod_{a=2}^k q_{+,a}^{s_{ai}} = s_i^+(q_+)q_{+,1}^{- \frac{s_{14}^+m_i}{m_4}}. \]
This together with (\ref{eq:MNchangeCase2}) implies
\[  \begin{aligned}
  s_i^-(q_-)x_-^{m_i}y_-^{-m_if_- + n_i^-} &= s_i^+(q_+)q_{+,1}^{- \frac{s_{14}^+m_i}{m_4}} x_-^{m_i}y_-^{-m_if_- - bm_i + n_i^+}\\
  & = s_i^+(q_+) \left(x_-q_{+,1}^{-\frac{s_{14}^+}{m_4}}\right)^{m_i}y_-^{-m_i(f_-+b)+n_i^+}.
\end{aligned} \]
The desired identification (\ref{eq:MirrorcurveWallCross}) thus follows from the relations
\begin{equation}\label{eq:XYFchangeCase2}
  x_- = x_+q_{+,1}^{\frac{s_{14}^+}{m_4}}, \qquad y_- = y_+, \qquad f_- = f_+-b.
\end{equation}

\begin{example} \rm{
As an example in Case IIa, consider the pair $\X_\pm$ that are related by a flop given in Figure \ref{fig:WallCrossCaseII} and the framed Aganagic-Vafa branes $(\L_\pm, f_\pm)$. The coordinates of $b_1, \dots, b_4$ with respect to the flags $(\tau, \sigma_+)$ and $(\tau, \sigma_-)$ are
\[  \begin{bmatrix}
      1 & 0 & 0 & 1\\
      0 & 1 & 0 & -1\\
      1 & 1 & 1 & 1
    \end{bmatrix}, \qquad
    \begin{bmatrix}
          1 & 0 & 0 & 1\\
          1 & 1 & 0 & 0\\
          1 & 1 & 1 & 1
        \end{bmatrix}
\]
respectively, and the change of basis matrix from $(\tau, \sigma_-)$ to $(\tau, \sigma_+)$ is
\[  \begin{bmatrix}
      1 & 0 & 0\\
      -1 & 1 & 0\\
      0 & 0 & 1
    \end{bmatrix}.  \]
The secondary fan of $\X_\pm$ is 1-dimensional, with $\enef(\X_+)$ spanned by $D_2 = D_4 = 1$ and $\enef(\X_-) = D_1 = D_3 = -1$. Take $p_1^\pm = \pm 1$. Then the mirror curve equations of $\X_\pm$ are
\[  H_+(q_+, x_+, y_+) = x_+y_+^{-f_+} + y_+ + 1 + q_{+,1}x_+y_+^{-f_+-1}, \]
\[ H_-(q_-, x_-, y_-) = q_{-,1}x_-y_-^{-f_-+1} + y_- + 1 + x_-y_-^{-f_-}, \]
which can be identified under the relations
\[  q_{-,1} = q_{+,1}^{-1}, \quad x_- = x_+q_{+,1}, \quad y_- = y_+, \quad f_- = f_++1.\]
For an example in Case IIb, see \cite{Fang19} for the identification of mirror curve equations of $K_{\P^2}$ and $[\C^3/\Z_3]$.
}\end{example}

\subsubsection{Proof of Proposition \ref{prop:MirrorcurveWallCross}, Case III}
In this case, we can identify basis vectors $v_j(\tau_+^1, \sigma_+^1) = v_j(\tau_-, \sigma_-)$ for $j = 1,2,3$, and thus coordinates
\[  m_i^+ = m_i(\tau_+^1, \sigma_+^1) = m_i(\tau_-, \sigma_-)= m_i^-, \qquad n_i^+=n_i(\tau_+^1, \sigma_+^1) = n_i(\tau_-, \sigma_-)= n_i^-  \]
for $i = 1,\dots, R$. Set $m_i = m_i^\pm, n_i = n_i^\pm$. Then $\ell = n_4, \ell_1 = n_2, \ell_2 = n_4 - n_2$. The short exact sequence (\ref{eq:StackyFanSq}) implies
\[  m_1D_1 + \sum_{i=5}^Rm_iD_i = 0, \qquad n_1D_1 + \ell_1D_2 + \ell D_4 + \sum_{i=5}^R n_iD_i = 0, \]
which implies
\begin{equation}\label{eq:DchangeCase3}
  D_4 = - \frac{\ell_1}{\ell} D_2 + \sum_{i = 5}^R \left(\frac{n_1m_i}{m_1 \ell} - \frac{n_i}{\ell} \right)D_i.
\end{equation}
Since $\enef(\X_+) \subseteq \enef_{\sigma_+^1} = \sum_{i = 4}^R \R_{\ge 0} D_i$ and $\enef(\X_-) \subseteq \enef_{\sigma_-} = \R_{\ge 0}D_2 + \sum_{i = 5}^R \R_{\ge 0} D_i$, we write
\[  p_1^+ = \sum_{i = 4}^R s_{1i}^+ D_i, \quad p_1^- = s_{12}^- + \sum_{i = 5}^R s_{1i}^- D_i, \quad p_a = s_{a2}D_2 + \sum_{i = 5}^R s_{ai}D_i \quad \mbox{for $a \ge 2$},  \]
and take $s_{11}^\pm = s_{12}^+ = s_{13}^\pm = s_{14}^- = 0$, $s_{ai} = 0$ for all $i = 1,3,4, a \ge 2$. From this and (\ref{eq:DchangeCase3}) we have
\[  p_1^+ = -\frac{s_{14}^+\ell_1}{\ell} D_2 + \sum_{i = 5}^R \left(s_{1i}^+ + \frac{s_{14}^+n_1m_i}{m_1\ell} - \frac{s_{14}^+n_i}{\ell}\right)D_i, \]
\[  -c_1p_1^- + c_2p_2 + \cdots + c_kp_k = \left(-c_1 s_{12}^- + \sum_{a = 2}^k c_a s_{a2} \right) D_2 + \sum_{i=5}^R \left(-c_1 s_{1i}^- + \sum_{a = 2}^k c_a s_{ai} \right) D_i. \]
Then (\ref{eq:POnePlus}) implies the relation
\[  \begin{bmatrix}
s_{12}^- & s_{22} & \cdots & s_{k2}\\
s_{15}^- & s_{25} & \cdots & s_{k5}\\
\vdots & \vdots & \ddots & \vdots\\
s_{1R}^- & s_{2R} & \cdots & s_{kR}
\end{bmatrix} \begin{bmatrix}
-c_1 \\ c_2 \\ \vdots \\ c_k
\end{bmatrix} = \begin{bmatrix}
  0 \\ s_{15}^+ \\ \vdots \\ s_{1R}^+
\end{bmatrix} + \frac{s_{14}^+n_1}{m_1\ell}\begin{bmatrix}
  0 \\ m_5 \\ \vdots \\ m_R
\end{bmatrix} - \frac{s_{14}^+}{\ell}\begin{bmatrix}
  \ell_1 \\ n_5 \\ \vdots \\ n_R
\end{bmatrix}.  \]
Adding in the relation (\ref{eq:Qchange}), we see that for each $i = 1, \dots, R$,
\[  \begin{aligned}
  s_i^-(q_-) = q_{-,1}^{s_{1i}^-}\prod_{a = 2}^k q_{-,a}^{s_{ai}} &= q_{+,1}^{-c_1s_{1i}^- + \sum_{a = 2}^k c_as_{ai}} \prod_{a=2}^k q_{+,a}^{s_{ai}} \\
  &= q_{+,1}^{s_{1i}^+ + \frac{s_{14}^+n_1m_i}{m_1\ell} - \frac{s_{14}^+n_i}{\ell}} \prod_{a=2}^k q_{+,a}^{s_{ai}} = s_i^+(q_+)q_{+,1}^{\frac{s_{14}^+n_1m_i}{m_1\ell} - \frac{s_{14}^+n_i}{\ell}}.
\end{aligned} \]
This implies
\[  \begin{aligned}
s_i^-(q_-)x_-^{m_i}y_-^{-m_if_- + n_i} &= s_i^+(q_+)q_{+,1}^{\frac{s_{14}^+n_1m_i}{m_1\ell} - \frac{s_{14}^+n_i}{\ell}}x_-^{m_i}y_-^{-m_if_- + n_i}\\
  &= s_i^+(q_+)\left(x_-q_{+,1}^{\frac{s_{14}^+n_1}{m_1\ell}- \frac{f_-s_{14}^+}{\ell}}\right)^{m_i}
  \left(y_- q_{+,1}^{-\frac{s_{14}^+}{\ell}}\right)^{-m_if_-+n_i}.
\end{aligned} \]
The desired identification (\ref{eq:MirrorcurveWallCross}) thus follows from the relations
\begin{equation}\label{eq:XYFchangeCase3}
  x_- = x_+q_{+,1}^{-\frac{s_{14}^+n_1}{m_1\ell}+ \frac{f_+s_{14}^+}{\ell}}, \qquad y_- = y_+ q_{+,1}^{\frac{s_{14}^+}{\ell}}, \qquad f_- = f_+.
\end{equation}

This completes the proof of Proposition \ref{prop:MirrorcurveWallCross}. Let us summarize the required relations (\ref{eq:XYFchangeCase1}, \ref{eq:XYFchangeCase2}, \ref{eq:XYFchangeCase3}) between parameters $(x_\pm, y_\pm, f_\pm)$ here:
\begin{equation}\label{eq:XYFchange}
  x_- = \begin{cases}
    x_+ & \mbox{ in Case I},\\
    x_+q_{+,1}^{\frac{s_{14}^+}{m_4}} & \mbox{ in Case II},\\
    x_+q_{+,1}^{-\frac{s_{14}^+n_1}{m_1\ell}+ \frac{f_+s_{14}^+}{\ell}} & \mbox{ in Case III},
  \end{cases} \qquad
\begin{matrix}
  y_-= \begin{cases}
    y_+ & \mbox{ in Cases I or II},\\
    y_+ q_{+,1}^{\frac{s_{14}^+}{\ell}} & \mbox{ in Case III},
  \end{cases} \vspace{1em}\\
  f_- = \begin{cases}
    f_+ & \mbox{ in Cases I or III},\\
    f_+-b & \mbox{ in Case II}.
  \end{cases}
\end{matrix}
\end{equation}

In Case III, we may further identify $H_+(q_+,x_+,y_+)$ with the mirror curve equation \\$H_{+,2}(q_+, x_{+,2}, y_{+,2})$ of $\X_+$ defined with respect to the flag $(\tau_+^2, \sigma_+^2)$ and framed Aganagic-Vafa brane $(\L_+^2, f_{+,2})$, as follows:

\begin{lemma}\label{lem:MirrorcurveIdCase3Plus}
There is a term-by-term identification
\begin{equation}\label{eq:MirrorcurveIdCase3Plus}
H_+(q_+, x_+, y_+) = y_+^{\ell_1}H_{+,2}(q_+, x_{+,2}, y_{+,2})
\end{equation}
under relations (\ref{eq:XYFchangeCase3Plus}) between parameters $(x_+, y_+)$ and $(x_{+,2}, y_{+,2})$ and framings $f_+$ and $f_{+,2}$.
\end{lemma}

\begin{proof}
On $\X_+$, we will change flags from $(\tau_+^2, \sigma_+^2)$ to $(\tau_+^1, \sigma_+^1)$. First observe that $s_4^-(q_-) = 1$. Thus the identification (\ref{eq:MirrorcurveWallCross}) and relations (\ref{eq:XYFchangeCase3}) imply that $s_4^+(q_+) = q_{+,1}^{s_{14}^+}$. The change of basis matrix from $\{v_1(\tau_+^2, \sigma_+^2)$, $v_2(\tau_+^2, \sigma_+^2), v_3(\tau_+^2, \sigma_+^2)\}$ to $\{v_1(\tau_+^1, \sigma_+^1), v_2(\tau_+^1, \sigma_+^1)$, $v_3(\tau_+^1, \sigma_+^1)\}$ has form
\[  \begin{bmatrix}
      1 & 0 & 0\\
      b' & 1 & \ell_1\\
      0 & 0 & 1
    \end{bmatrix}
\]
for some $b' \in \Z$. In particular, $m_1(\tau_+^2,\sigma_+^2) = m_1$ and $n_1(\tau_+^2, \sigma_+^2) = n_1 - bm_1 - \ell_1$. The desired identification (\ref{eq:MirrorcurveIdCase3Plus}) will follow from (\ref{eq:MirrorcurveFlagChange}) if we use the following reparametrizations (see (\ref{eq:XYFlagChange}, \ref{eq:MirrorcurveXYReparam})):
\[  x_{+,2}y_{+,2}^{-f_{+,2}} = x_+y_+^{-f_++b'}q_{+,1}^{\frac{s_{14}^+(bm_1-n_1+\ell_1)}{m_1\ell_2}}, \qquad y_{+,2} = y_+q_{+,1}^{\frac{s_{14}^+}{\ell_2}}, \]
which is equivalent to
\begin{equation}\label{eq:XYFchangeCase3Plus}
x_{+,2} = x_+q_{+,1}^{\frac{s_{14}^+(\ell_1 - n_1)}{m_1\ell_2} + \frac{s_{14}^+f_+}{\ell_2}}, \qquad y_{+,2} = y_+ q_{+,1}^{\frac{s_{14}^+}{\ell_2}}, \qquad f_{+,2} = f_+ - b'. \qedhere
\end{equation}
\end{proof}

\begin{example} \label{ex:A1MCId} \rm{
Consider the crepant resolution $\X_+ = K_{\P^1} \oplus \O_{\P^1} \to  \A_1 = \X_-$ (see Examples \ref{ex:A1SecondaryFan} and \ref{ex:A1MC}). Take a framed Aganagic-Vafa brane $(\L_-, f_-)$ in $\A_1$ that intersects the cone spanned by $b_3$ and $b_4$, and let $(\L_+^1, f_+), (\L_+^2, f_{+,2})$ be frame Aganagic-Vafa branes in the preimage. Take $p_1^+ = D_3 = D_4 = 1$ and $p_1^- = D_2 = -2$. Then the mirror curve equations
\[  H_+(q_+, x_+, y_+) = x_+y_+^{-f_+} + y_+ + 1 + q_{+,1}y_+^2,  \]
\[  H_{+,2}(q_+, x_{+,2}, y_{+,2}) = x_{+,2}y_{+,2}^{-f_{+,2}} + 1 + q_{+,1}y_{+,2}^{-1} + y_{+,2}, \]
\[  H_-(q_-, x_-, y_-) = x_-y_-^{-f_-} + q_{-,1}y_- + 1 + y_-^2 \]
satisfy the identification $H_- = H_+ = y_+H_{+,2}$ under the relations
\[  q_{-,1} = q_{+,1}^{-\frac{1}{2}}, \quad x_-q_{+,1}^{-f_+} = x_+ = x_{+,2}q_{+,1}^{-f_+-1}, \quad y_-q_{+,1}^{-\frac{1}{2}} = y_+ = y_{+,2}q_{+,1}^{-1}, \quad f_- = f_+ = f_{+,2}-1.  \]
}\end{example}

\begin{remark}\label{rem:} \rm{
An alternative interpretation of the framing relations (\ref{eq:XYFchange}, \ref{eq:XYFchangeCase3Plus}) is that, these are the relations required to preserve the 1-dimensional framing torus $T_f'$ defined at the end of Section \ref{sect:AVBranes} under the crepant transformations. In other words, our toric crepant transformations are all $T_f'$-equivariant.
}
\end{remark}

\section{Global Mirror Curve and Analytic Continuation}\label{sect:GlobalMirrorCurve}
In this section, we prove the Open Crepant Transformation Conjecture through focusing on a pair of toric Calabi-Yau 3-orbifolds $\X_+$ and $\X_-$ which differ by a single wall-crossing in the secondary fan (Theorem \ref{thm:OCTC}). The mirror theorem of Fang-Liu-Tseng \cite{FLT19} (Theorem \ref{thm:MirrorThmDisk}) relates disk invariants of $\X_\pm$ to local solutions to their mirror curve equations in the large radius limit. Our main observation is that using the identification in Proposition \ref{prop:MirrorcurveWallCross}, the mirror curves of $\X_\pm$ fit into a global mirror curve equation over an open B-model moduli space. Our desired identification of disk invariants can then be achieved via analytic continuation on the local solutions to the global mirror curve equation. We use the notations introduced in Sections \ref{sect:CrepantDisk} and \ref{sect:MirrorcurveWallCross}.

\subsection{Open B-model moduli space}\label{sect:OpenModuli}
We start with the construction of the open B-model moduli space. We treat the two cases where the Aganagic-Vafa branes are outer (resp. inner) separately.

\subsubsection{Construction in the outer case}
Let $\M_\pm = \Spec \C[q_{\pm,1}, \dots, q_{\pm,k}] \cong \C^k$. We glue the two charts together using the relation (\ref{eq:Qchange}) on the common open subset $\{q_{\pm,1} \neq 0\}$ to obtain a $k$-dimensional space $\M$. Through Proposition \ref{prop:MirrorcurveWallCross}, the mirror curve $C_{q_\pm}^\pm$ of $\X_\pm$, defined over the chart $\M_\pm$, fit into a global family of affine curves over $\M$.

\begin{remark} \rm{
There is a map from $\M_\pm$ to the affine chart corresponding to $\X_\pm$ in the secodary variety given by the lattice map $\Z^k \to \Lat^\vee$ that sends the first basis vector to $p_{\pm, 1}$ and the $a$-th basis vector ($a \ge 2$) to $p_a$. Under this map, we obtain a compactification of the family $C_{q_{\pm}}^\pm$ as the pullback of a divisor in a flat family of toric surfaces over the corresponding chart of the secondary variety. See \cite[Section 4.3.1]{FLT19} and the references therein.
}\end{remark}

\begin{example} \label{ex:A1GlobalMC} \rm{
For the crepant resolution $K_{\P^1} \oplus \O_{\P^1} \to  \A_1$ (see Examples \ref{ex:A1SecondaryFan}, \ref{ex:A1MC}, and \ref{ex:A1MCId}), $\M = \P(1,2)$ is the secondary variety. The compactified global mirror curve over $\M$ is illustrated in Figure \ref{fig:A1GlobalMC}.
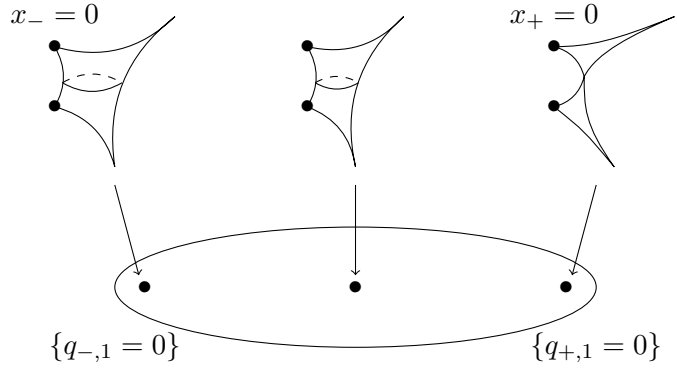
\begin{figure}[htb]
  \[
      \begin{tikzpicture}[scale=.8]
        \draw (0,0) ellipse (4 and 1);
        \node at (-3.5, 0) {$\bullet$};
        \node at (-4, -1) {$\{q_{-,1} =0\}$};
        \node at (3.5, 0) {$\bullet$};
        \node at (4,-1) {$\{q_{+,1} = 0\}$};
        \node at (0,0) {$\bullet$};

        \draw (-5,4) to[bend left] (-5,3) to[bend left] (-4, 2) to[bend left] (-3, 4.5) to[bend left] (-5,4);
        \node at (-5, 4) {$\bullet$};
        \node at (-5, 3) {$\bullet$};
        \node at (-5, 4.5) {$x_- = 0$};
        \draw[->] (-4, 1.7) -- (-3.6,0.2);
        \draw (-4.87, 3.4) to[bend right] (-3.88, 3.4);
        \draw[dashed] (-4.87, 3.4) to[bend left] (-3.88, 3.4);

        \draw (-0.8,4) to[bend left] (-0.8,3) to[bend left] (0, 2) to[bend left] (0.8, 4.5) to[bend left] (-0.8,4);
        \node at (-0.8, 4) {$\bullet$};
        \node at (-0.8, 3) {$\bullet$};
        \draw[->] (0, 1.7) -- (0,0.2);
        \draw (-0.67, 3.4) to[bend right] (0.05, 3.4);
        \draw[dashed] (-0.67, 3.4) to[bend left] (0.05, 3.4);

        \coordinate (1) at (3.8,3.5);
        \draw (1) to[bend right] (3.3, 4) .. controls (3.9, 4) and (4.3, 4.1) .. (5.3,4.5) .. controls (4.6, 4.2) and (4, 3.9) ..  (1) to[bend left] (3.3,3) .. controls (3.8, 2.6) .. (4.3, 2) .. controls (3.8, 2.7) and (3.8, 3).. (1);
        \node at (3.3, 4) {$\bullet$};
        \node at (3.3, 3) {$\bullet$};
        \node at (3.3, 4.5) {$x_+ = 0$};
        \draw[->] (4, 1.7) -- (3.6,0.2);
      \end{tikzpicture}
  \]
  \caption{The compactified global mirror curve of $K_{\P^1} \oplus \O_{\P^1} \to \A_1$. All curves in the family have genus $0$. The curve over the \emph{orbifold point} $\{q_{-,1} = 0\}$ or a generic point in $\M$ is smooth, while the curve over the \emph{large radius limit point} $\{q_{+,1}=0\}$ is nodal and has two components. On each curve, there are two points with $x=0$, which corresponds to the open large radius limits to be introduced below.}
  \label{fig:A1GlobalMC}
\end{figure}
The global mirror curve for the crepant resolution $K_{\P^2} \to [\C^3/\Z_3]$ is studied in \cite{Fang19}.
}\end{example}

We add an extra dimension to $\M$ to incorporate the open moduli parameters and the framing dependence. Let $\tilde{\M}_\pm = \Spec \C[q_{\pm,1}, \dots, q_{\pm, k}, x_\pm] \cong \C^{k+1}$. The \emph{open B-model moduli space} $\tilde{\M}$ is formed by gluing the two charts $\tilde{\M}_\pm$ together using the relations (\ref{eq:Qchange}, \ref{eq:XYFchange}) on the common open subset $\{q_{\pm,1} \neq 0\}$. Note that $\M = \{x_\pm=0\} \subset \tilde{\M}$. We call the point $P_\pm \in \tilde{\M}_\pm$ where $(q_\pm, x_\pm) = 0$ the \emph{large radius limit (LRL)} point of $\X_\pm$, which represents a large complex structure/orbifold/open mixed-type limit. This is consistent with the limit $q_\pm \to 0, x_\pm \to 0$ we took in Section \ref{sect:MirrorThmDisk}.

As in (\ref{eq:MirrorcurveIncludeX}), we can view $H_\pm(q_\pm, x_\pm, y_\pm)=0$ as a family of equations in the indeterminate $y_\pm$ parametrized by $(q_\pm, x_\pm)$. Since the Aganagic-Vafa branes are outer, as we observed in Section \ref{sect:MirrorcurveOpen}, each coefficient $\tilde{s}_i^\pm(q_\pm, x_\pm)$ in $H_\pm$ (see (\ref{eq:CoeffWallCrossForm})) is a monomial in $(q_\pm, x_\pm)$. Thus this family is defined over all of $\tilde{\M}_\pm$. Proposition \ref{prop:MirrorcurveWallCross} then implies that the two families fit into a global family of equations over $\tilde{\M}$.

\subsubsection{Construction in the inner case}
Here the construction of the open B-model moduli space $\tilde{\M}$ is similar to the outer case, except that we would still like to obtain a global family of mirror curve equations defined over all of $\tilde{\M}$. Let
\[  q_\pm^\alpha = \begin{cases}
      s_4^\pm(q_\pm)^{-\frac{1}{m_4}} & \mbox{ in Case I },\\
      s_5^\pm(q_\pm)^{-\frac{1}{m_5}} & \mbox{ in Cases II or III},
    \end{cases} \]
be as defined in Lemma \ref{lem:MCNegCoeff}. As we observed after Lemma \ref{lem:MCNegCoeff}, if we pick $a_0 \in A_K(\X_-) \subseteq A_K(\X_+)$ such that $q_{-,a_0}$ divides $q_-^\alpha$, then each $\tilde{s}_i^-(q_-, x_-)$ is a monomial in $q_{-,1}, \dots, q_{-, a_0-1}, q_{-, a_0+1}, \dots, q_{-,k}$, $x_-, q_{-,a_0}x_-^{-1}$.

\begin{lemma}\label{lem:InnerSpecialA0}
There exists $a_0 \in A_K(\X_-)$, $a_0 \neq 1$, such that $q_{-,a_0}$ divides $q_-^\alpha$.
\end{lemma}

\begin{proof}
In Case I, observe that $s_4^+(q_+) = s_4^-(q_-)$ under the relation (\ref{eq:Qchange}), which implies that $q_+^\alpha = q_-^\alpha$. Since both $q_\pm^\alpha$ are monomials, and $q_{-,1} = q_{+,1}^{-c_1}$ for $c_1 >0$, $q_-^\alpha$ cannot be a power of $q_{-,1}$ and must contain some other $q_{-,a_0}$ as a factor.

In Case IIa, note that $s_1^+ = 1$. From the identification (\ref{eq:MirrorcurveWallCross}) and the relation (\ref{eq:XYFchangeCase2}), we see that $s_1^-$ is a power of $q_{-,1}$. Since $p_{-,1}, \dots, p_{-,k}$ are linearly independent, $s_5^-$ cannot be a power of $q_{-,1}$ as well and must contain some other $q_{-,a_0}$ as a factor. Thus $q_{-,a_0}$ divides $q_-^\alpha$.

In Cases IIb or III, we have $p_{-,1} = D_1$. Thus $1 \not \in A_K(\X_-)$, and the lemma follows.
\end{proof}

With $a_0$ chosen as in Lemma \ref{lem:InnerSpecialA0}, the identification (\ref{eq:MirrorcurveWallCross}) and the relation (\ref{eq:Qchange}) imply that $q_{+,a_0}$ divides $q_+^\alpha$, so that each $\tilde{s}_i^+(q_+, x_+)$ is a monomial in $q_{+,1}, \dots, q_{+, a_0-1}, q_{+, a_0+1} ,\dots, q_{+,k}$, $x_+, q_{+,a_0}x_+^{-1}$. As a shorthand notation, we set $\hat{q}_{\pm} := (q_{\pm,1}, \dots, q_{\pm, a_0-1}, q_{\pm, a_0+1}, \dots, q_{\pm, k})$. Moreover, we set
\begin{equation}\label{eq:ZXRelation}
z_\pm := q_{\pm, a_0}x_\pm^{-1}.
\end{equation}
Then $z_\pm$ have relation
\begin{equation}\label{eq:Zchange}
  z_- = \begin{cases}
    z_+q_{+,1}^{c_{a_0}} & \mbox{ in Case I},\\
    z_+q_{+,1}^{c_{a_0} - \frac{s_{14}^+}{m_4}} & \mbox{ in Case II},\\
    z_+q_{+,1}^{c_{a_0}+\frac{s_{14}^+n_1}{m_1\ell}- \frac{f_+s_{14}^+}{\ell}} & \mbox{ in Case III}.
  \end{cases}
\end{equation}

Let $\tilde{\M}_\pm = \Spec \C[q_{\pm,1}, \dots, q_{\pm, a_0-1}, q_{\pm, a_0+1}, q_{\pm, k}, x_\pm, z_\pm] \cong \C^{k+1}$. The \emph{open B-model moduli space} $\tilde{\M}$ in the inner case is formed by gluing the two charts $\tilde{\M}_\pm$ together using the relations (\ref{eq:Qchange} ,\ref{eq:XYFchange}, \ref{eq:Zchange})
on the common open subset $\{q_{\pm,1} \neq 0\}$. We call the point $P_\pm \in \tilde{\M}_\pm$ where $(\hat{q}_\pm, x_\pm, z_\pm) = 0$ the \emph{large radius limit (LRL)} point of $\X_\pm$, which is consistent with the limit $q_\pm \to 0, x_\pm \to 0, q_\pm^\alpha x_\pm^{-1} \to 0$ we took in Section \ref{sect:MirrorThmDisk}.

As in (\ref{eq:MirrorcurveIncludeX}), we can view $H_\pm(q_\pm, x_\pm, y_\pm)=0$ as a family of equations in the indeterminate $y_\pm$ parametrized by $(\hat{q}_\pm, x_\pm, z_\pm)$.
Since the Aganagic-Vafa branes are inner, as we observed in Section \ref{sect:MirrorcurveOpen}, each coefficient $\tilde{s}_i^\pm(q_\pm, x_\pm)$ in $H_\pm$ (see (\ref{eq:CoeffWallCrossForm})) is a monomial in $(\hat{q}_\pm, x_\pm, z_\pm)$. Thus this family is defined over all of $\tilde{\M}_\pm$. Proposition \ref{prop:MirrorcurveWallCross} then implies that the two families fit into a global family of equations over $\tilde{\M}$.

\begin{remark} \rm{
In either the outer or the inner case, $\widetilde{\M}$ is isomorphic to the total space of a rank-$k$ vector bundle over a weighted projective line, and $q_{\pm,1}$ give coordinates on the base.
}\end{remark}

\subsection{Proof of Theorem \ref{thm:OCTC}}\label{sect:OCTCProof}
Recall that the mirror theorem for disk invariants (Theorem \ref{thm:MirrorThmDisk}) relates the disk potentials to local solutions to the mirror curve equation around the LRL points. Via Proposition \ref{prop:MirrorcurveWallCross}, we showed in the previous subsection that the mirror curve equations $H_\pm = 0$ of $\X_\pm$ fit into a global family over the open B-model moduli space $\tilde{\M}$. We will identify disk invariants of $\X_\pm$ through analytic continuation of local solutions to the global mirror curve equation on $\tilde{\M}$.

We start by defining a particular coordinate subspace of $\tilde{\M}_\pm$. For $j = 1, \dots, \ell-1$, let $i_j \in I_\orb(\X_-)$ such that $m_{i_j} = 0$ and $n_{i_j}^\pm = j$. Then each $b_{i_j}$ is an interior lattice point of the cone $\tau$ (or $\tau_-$ in Case III). Recall the bijection $\iota: I_\orb(\X_-) \to A_\orb(\X_-)$ from condition \ref{cond:PaOrbChoice} in Section \ref{sect:MirrorCurve}. Set $a_j = \iota(i_j)$ and
\[  A_0 = \{a_1, \dots, a_{\ell-1}\} \subseteq A_\orb(\X_-).  \]
Then the mirror curve equation $H_-$ has form
\begin{equation}\label{eq:HMinusSimplified}
  1 + q_{-, a_1}y_- + \cdots + q_{-, a_{\ell-1}}y_-^{\ell-1} + y_-^\ell + \cdots
\end{equation}
where each of the remaining terms contains as a factor $q_{-, a}$ for some $a \not \in A_0$, $x_-$ (or $z_-$). Note that in the inner case, $a_0 \not \in A_0$. $A_0$ defines an $(\ell-1)$-dimensional coordinate subspace $\tilde{\M}_{\pm,0}$ of $\tilde{\M}_\pm$:
\[
\begin{aligned}
   \tilde{\M}_{\pm,0} &= \Spec \C[q_{\pm, a_1}, \dots, q_{\pm, a_{\ell-1}}]\\
    &= \begin{cases}
  \{q_{\pm, a} = 0 \mbox{ for all } a \not \in A_0, x_\pm = 0\} & \mbox{ in the outer case},\\
  \{q_{\pm, a} = 0 \mbox{ for all } a \not \in A_0 \cup \{a_0\}, x_\pm = z_\pm = 0\} & \mbox{ in the inner case}.
\end{cases}
\end{aligned} \]

As in Section \ref{sect:MirrorThmDisk}, if $\L_-$ is outer, there are $\ell$ solutions $\kappa_1^-, \dots, \kappa_\ell^-$ to the equation $H_-(y_-) = 0$ in a neighborhood of the LRL point $P_-$ that satisfy
\[  \lim_{q_- \to 0, x_- \to 0} \log \kappa_j^- = \frac{\pi\sqrt{-1}}{\ell}(-1+2j). \]
These solutions have local power series expansions in $(q_-, x_-)$. If $\L_-$ is inner, there are $\ell$ solutions $\kappa_1^-, \dots, \kappa_\ell^-$ to the equation $H_-(y_-) = 0$ in a neighborhood of $P_-$ that satisfy
\[  \lim_{\hat{q}_- \to 0, x_- \to 0, z_- \to 0} \log \kappa_j^- = \frac{\pi\sqrt{-1}}{\ell}(-1+2j). \]
These solutions have local power series expansions in $(\hat{q}_-, x_-, z_-)$, which via (\ref{eq:ZXRelation}) translate to power series in $(\hat{q}_-, x_-, q_-^\alpha x_-^{-1})$. Depending on the framing $f_-$, $H_-(y_-)=0$ may have other solutions near $P_-$, but these solutions all have a pole along $\tilde{\M}_{-,0}$. Theorem \ref{thm:MirrorThmDisk} gives the relation
\begin{equation}\label{eq:WVMinus}
  \left( x_-\frac{\partial}{\partial x_-} \right)^2
  \begin{bmatrix}   \xi_\ell^{\ell-1}W_1^-\\ \vdots \\ \xi_\ell W_{\ell-1}^- \\ W_\ell^- \end{bmatrix}
    = U_\ell \left( x_-\frac{\partial}{\partial x_-} \right)
    \begin{bmatrix} \log \kappa_1^-\\ \vdots \\ \log \kappa_{\ell-1}^- \\ \log \kappa_\ell^- \end{bmatrix}.
\end{equation}

Let us now turn to $\X_+$. We first consider Cases I or II. Similar to above, there are $\ell$ solutions $\kappa_1^+, \dots, \kappa_\ell^+$ to the equation $H_+(y_+)=0$ in a neighborhood of the LRL point $P_+$ that satisfy
\[  \lim_{\substack{q_+ \to 0, x_+ \to 0 \\ \mbox{or } \hat{q}_+ \to 0, x_+ \to 0, z_+ \to 0}} \log \kappa_j^+ = \frac{\pi\sqrt{-1}}{\ell}(-1+2j). \]
Depending on the framing $f_+$, $H_+(y_+) = 0$ may have other solutions near $P_+$, but these solutions all have a pole along $\tilde{\M}_{+,0}$. Theorem \ref{thm:MirrorThmDisk} gives the relation
\begin{equation}\label{eq:WVPlusCase12}
  \left( x_+\frac{\partial}{\partial x_+} \right)^2
  \begin{bmatrix}   \xi_\ell^{\ell-1}W_1^+\\ \vdots \\ \xi_\ell W_{\ell-1}^+ \\ W_\ell^+ \end{bmatrix}
    = U_\ell \left( x_+\frac{\partial}{\partial x_+} \right)
    \begin{bmatrix} \log \kappa_1^+\\ \vdots \\ \log \kappa_{\ell-1}^+ \\ \log \kappa_\ell^+ \end{bmatrix}.
\end{equation}

Now along a path in $\tilde{\M}$ from $P_+$ to $P_-$, the local solutions $\kappa_1^+, \dots, \kappa_\ell^+$ analytically continue to a permutation of $\kappa_1^-, \dots, \kappa_\ell^-$. Using Lemma \ref{lem:Monodromy} in the subsequent subsection, we are able find a path along which $\kappa_1^+, \dots, \kappa_\ell^+$ analytically continue to $\kappa_1^-, \dots, \kappa_\ell^-$ in the prescribed order. From the relation (\ref{eq:XYFchange}) between $x_\pm$, we see that
\[  x_+\frac{\partial}{\partial x_+} = x_-\frac{\partial}{\partial x_-}.  \]
Thus (\ref{eq:WVMinus}, \ref{eq:WVPlusCase12}) imply that
\[  \left( x_+\frac{\partial}{\partial x_+} \right)^2
\begin{bmatrix}   \xi_\ell^{\ell-1}W_1^+\\ \vdots \\ \xi_\ell W_{\ell-1}^+ \\ W_\ell^+ \end{bmatrix} = \left( x_-\frac{\partial}{\partial x_-} \right)^2
\begin{bmatrix}   \xi_\ell^{\ell-1}W_1^-\\ \vdots \\ \xi_\ell W_{\ell-1}^- \\ W_\ell^- \end{bmatrix}  \]
under analytic continuation. The desired identification (\ref{eq:OCTCThmCase12}) then follows since each term in each series $W_j^\pm$ contains a nonzero power of $x_\pm$ as a factor.

Now we consider Case III. In this case we have $i_{\ell_1} = 1$ and $a_{\ell_1} = \iota(1) = 1$. Thus $\tilde{\M}_{\pm,0}$ glue to form a (connected) $(\ell-1)$-dimensional coordinate subspace $\tilde{\M}_0$ of $\tilde{\M}$. Since $s_4^+(q_+) = q_{+,1}^{s_{14}^+}$ as we showed in the proof of Lemma \ref{lem:MirrorcurveIdCase3Plus}, we observe via Lemma \ref{lem:MCOrbCoeff} that the mirror curve equation $H_+$ has form
\[  1 + q_{+,a_1}y_+ \cdots + q_{+, a_{\ell_1-1}}y_+^{\ell_1-1} + y_+^{\ell_1} + q_{+, a_{\ell_1+1}}q_{+,1}^{\frac{s_{14}^+}{\ell_2}}y_+^{\ell_1+1} + \cdots + q_{+, a_{\ell-1}}q_{+,1}^{\frac{(\ell_2-1)s_{14}^+}{\ell_2}}y_+^{\ell-1} + q_{+,1}^{s_{14}^+}y_+^{\ell} + \cdots \]
where each of the remaining terms contains as a factor $q_{+,a}$ for some $a \not \in A_0$, $x_+$ (or $z_+$).

Similar to above, there are $\ell_1$ solutions $\kappa_1^+, \dots, \kappa_{\ell_1}^+$ to the equation $H_+(y_+) = 0$ in a neighborhood of the LRL point $P_+$ that satisfy
\[  \lim_{\substack{q_+ \to 0, x_+ \to 0 \\ \mbox{or } \hat{q}_+ \to 0, x_+ \to 0, z_+ \to 0}} \log \kappa_j^+ = \frac{\pi\sqrt{-1}}{\ell_1}(-1+2j). \]
Moreover, there are $\ell_2$ solutions $\kappa_{\ell_1+1}^+, \dots, \kappa_{\ell}^+$ to $H_+(y_+)=0$ locally near $P_+$ that satisfy the asymptotics
\[  \log \kappa_{\ell_1+j}^+ \sim \frac{\pi\sqrt{-1}}{\ell_2}(-1+2j) - \frac{s_{14}^+}{\ell_2}\log q_{+,1}.  \]
Depending on the framing $f_+$, $H_+(y_+)=0$ may have other solutions near $P_+$, by these solutions all have a pole along $\tilde{\M}_{+,0}$.

From Lemma \ref{lem:MirrorcurveIdCase3Plus} and especially the relation $y_{+,2} = y_+q_{+,1}^{\frac{s_{14}^+}{\ell_2}}$ from (\ref{eq:XYFchangeCase3Plus}), we see that
$q_{+,1}^{\frac{s_{14}^+}{\ell_2}}\kappa_{\ell_1+1}$, $\dots$, $q_{+,1}^{\frac{s_{14}^+}{\ell_2}}\kappa_\ell$ are the $\ell_2$ solutions to $H_{+,2}(y_{+,2}) = 0$ that converges at $P_+$. Moreover, the relation $x_{+,2} = x_+q_{+,1}^{\frac{s_{14}^+(\ell_1 - n_1)}{m_1\ell_2} + \frac{s_{14}^+f_+}{\ell_2}}$ from (\ref{eq:XYFchangeCase3Plus}) implies that
\[  x_{+,2} \frac{\partial}{\partial x_{+,2}} = x_+ \frac{\partial}{\partial x_+}.  \]
Thus, if we view each $W_{\ell_1+j}^+$ as a series in $(q_+, x_+)$,  Theorem \ref{thm:MirrorThmDisk} gives the relation
\begin{equation}\label{eq:WVPlusCase3}
  \left( x_+\frac{\partial}{\partial x_+} \right)^2\begin{bmatrix}   \xi_{\ell_1}^{\ell_1-1}W_1^+\\ \vdots \\ \xi_{\ell_1} W_{\ell_1-1}^+ \\ W_{\ell_1}^+ \\ \xi_{\ell_2}^{\ell_2-1} W_{\ell_1+1}^+ \\ \vdots \\ W_\ell^+ \end{bmatrix} = \diag(U_{\ell_1}, U_{\ell_2}) \left( x_+\frac{\partial}{\partial x_+} \right)
    \begin{bmatrix} \log \kappa_1^+\\ \vdots \\ \log \kappa_{\ell-1}^+ \\ \log \kappa_\ell^+ \end{bmatrix}.
\end{equation}

Now along a path in $\tilde{\M}_0$ from $P_+$ to $P_-$, the local solutions $\kappa_1^+, \dots, \kappa_\ell^+$ analytically continue to a permutation of $\kappa_1^-, \dots, \kappa_\ell^-$, since the two sets of solutions are the only ones that do not have a pole along $\tilde{\M}_0$. Again we use Lemma \ref{lem:Monodromy} to select a path along which $\kappa_1^+, \dots, \kappa_\ell^+$ analytically continue to $\kappa_1^-, \dots, \kappa_\ell^-$ in the prescribed order. Then the desired identification (\ref{eq:OCTCThmCase3}) follows from (\ref{eq:WVMinus}, \ref{eq:WVPlusCase3}) in a way similar to the previous case.

\subsection{Monodromy of solutions}\label{sect:Monodromy}
In this subsection, we prove a lemma on the monodromy of solutions to the mirror curve equation that is required by the analytic continuation in the proof of Theorem \ref{thm:OCTC}. The proof of the lemma is due to Iritani.

\begin{lemma}\label{lem:Monodromy}
Consider the coordinate subspace $\tilde{\M}_{-,0} = \Spec \C[q_{-,a_1}, \dots, q_{-,a_{\ell-1}}]$ with origin $P_-$ and the local solutions $\kappa_1^-, \dots$, $\kappa_\ell^-$ to the mirror curve equation $H_-(y_-)=0$ as defined in the previous subsection. Given any permutation $\sigma \in S_\ell$, there exist a loop $\gamma_\sigma$ in $\tilde{\M}_{-,0}$ based at $P_-$ along which $\kappa_1^-, \dots, \kappa_\ell^-$ analytically continue to $\kappa_{\sigma(1)}^-, \dots, \kappa_{\sigma(\ell)}^-$ in the prescribed order.
\end{lemma}

\begin{proof}
From (\ref{eq:HMinusSimplified}), we see that on $\tilde{\M}_{-,0}$, $\kappa_1^-, \dots, \kappa_\ell^-$ are solutions to the restricted equation
\begin{equation}\label{eq:HMinusRestricted}
1 + q_{-,a_1}y_- + \cdots + q_{-,a_{\ell-1}}y_-^{\ell-1}+y_-^\ell = 0.
\end{equation}
In particular, $\kappa_1^- \cdots \kappa_\ell^- = (-1)^\ell$, and if $s_j$ is the $j$-th elementary symmetric polynomial on $\ell$ variables, then
\[ q_{-, a_{\ell-j}} = (-1)^js_j(\kappa_1^-, \dots, \kappa_\ell^-). \]
Now consider the action of $S_\ell$ on $\hat{R} = \C[\kappa_1^-, \dots, \kappa_\ell^-]/(\kappa_1^- \cdots \kappa_\ell^- - (-1)^\ell)$ by permuting the $\kappa_j^-$'s. The injective ring homomorphism
\[  \C[q_{-,a_1}, \dots, q_{-,a_{\ell-1}}] \to \hat{R}, \qquad q_{-, a_{\ell-j}} \mapsto (-1)^js_j(\kappa_1^-, \dots, \kappa_\ell^-)  \]
has image equal to the $S_\ell$-invariant subring of $\hat{R}$. This implies that the induced map
\[  h: \hat{\M}_- := \Spec \hat{R} \to \tilde{\M}_{-,0}   \]
is a ramified $S_\ell$-cover. The preimage of $h^{-1}(P_-)$ consists of $\ell!$ points $\{(\kappa_{\sigma(1)}^-(0), \dots, \kappa_{\sigma(\ell)}^-(0)) \mid \sigma \in S_\ell\}$. Note that $\hat{\M}_-$ is an irreducible hypersurface in $\C^\ell = \Spec \C[\kappa_1^-, \dots, \kappa_\ell^-]$ and is thus connected. If we take a path $\hat{\gamma}_\sigma$ from $(\kappa_1^-(0), \dots, \kappa_\ell^-(0))$ to $(\kappa_{\sigma(1)}^-(0), \dots, \kappa_{\sigma(\ell)}^-(0))$ in $\widehat{\M}_-$, the projection $\gamma_\sigma = h \circ \hat{\gamma}_\sigma$ gives our desired loop.
\end{proof}

\begin{remark} \label{ref:DiscriminantLocus} \rm{
Lemma \ref{lem:Monodromy} is already used in \cite{CCIT09}, where it is stated as: the monodromy around the discriminant locus of (\ref{eq:HMinusRestricted}) in $\tilde{\M}_{-,0}$, i.e. the locus where (\ref{eq:HMinusRestricted}) has repeated roots, acts transitively on the roots $\kappa_1^-, \dots, \kappa_n^-$ by permutation (see the proof of Proposition A.7). Indeed, the $S_\ell$-cover $h$ defined above is ramified over the big diagonal of $\hat{\M}_-$, i.e. the set of points $(\kappa_1^-, \dots, \kappa_\ell^-)$ where not all coordinates are distinct, and the image of the big diagonal under $h$ is exactly the discriminant locus of (\ref{eq:HMinusRestricted}). We can restrict $h$ to the complement of the big diagonal in $\hat{\M}_-$ and the complement of the discriminant locus of (\ref{eq:HMinusRestricted}) in $\tilde{\M}_{-,0}$, and the same argument goes through.
} \end{remark}


\end{document}